\documentclass[11pt,letterpaper]{article}

\usepackage{stylefile}
\usepackage{fullpage}
\usepackage[colorinlistoftodos,bordercolor=orange,backgroundcolor=orange!20,linecolor=orange,textsize=scriptsize]{todonotes}

\title{Coordinate Descent with Arbitrary Sampling II:\\Expected Separable Overapproximation\thanks{The  authors  acknowledge support from the EPSRC  Grant EP/K02325X/1,
{\em Accelerated Coordinate Descent Methods for Big Data Optimization}. Most of the material of this paper was obtained by the authors in Spring 2014, and was presented by PR  in June 2014 at the  ``Khronos Days Summer School'' focused on
``High-Dimensional Learning and Optimization'' in Grenoble, France \cite{Grenoble_theory}; \url{http://www.maths.ed.ac.uk/\%7Eprichtar/docs/cdm-talk.pdf}.}}

\author{Zheng Qu \footnote{School of Mathematics, The University of Edinburgh, United Kingdom (e-mail: zheng.qu@ed.ac.uk)} \qquad \qquad 
 Peter Richt\'{a}rik \footnote{School of Mathematics, The University of Edinburgh, United Kingdom (e-mail: peter.richtarik@ed.ac.uk) }}

\begin{document}
\maketitle

\begin{abstract} The design and complexity analysis of randomized coordinate descent methods, and in particular of variants which update a random subset (sampling) of coordinates in each iteration, depends on the notion of  expected separable overapproximation (ESO). This refers to an inequality involving the objective function and the sampling,  capturing in a compact way certain smoothness properties of the function in a random subspace spanned by the sampled coordinates. ESO inequalities were  previously established for special classes of samplings only, almost invariably for uniform samplings. In this paper we develop a systematic technique for deriving these inequalities for a large class of functions  and for {\em arbitrary samplings}. We demonstrate that one can recover existing ESO results using our general approach, which is based on the study of  eigenvalues associated with samplings and the data describing the function. 

\end{abstract}

\section{Introduction}

Coordinate descent  methods have been popular with practitioners for many decades due to their inherent conceptual simplicity and ease with which one can produce a working code. However, up to a  few exceptions \cite{Tseng:CBCDM:Nonsmooth,Tseng:CCMCDM:Smooth}, they have been largely ignored in the optimization community until recently when a renewed interest in  coordinate descent was sparked by several reports of  their remarkable success in certain applications \cite{Protein2003,WuLange:2008,TTD}. Additional and perhaps more significant reason behind the recent flurry of research activity in the area of coordinate descent comes from  breakthroughs in our theoretical understanding of these methods through the introduction of {\em randomization} in the iterative process \cite{Nesterov12,RT:SPARS11,UCDC, PCDM,SDCA,Pegasos2,ICD,DQA,SPCDM,Hydra,Hydra2,Necoara:rcdm-coupled,RBCDMLuXiao,lee2013efficient, APPROX,APPROX-SC,APCG,WrightAsynchrous13,WrightAsynchrous14,QUARTZ,Paper1,UNI}. Traditional variants of coordinate descent rely on cyclic or greedy rules for the selection of the next coordinate to be updated.

\subsection{Expected Separable Overapproximation} \label{sec:introESO}
It has recently become increasingly clear that the design and complexity analysis of  randomized coordinate descent methods is intimately linked with and can be better understood through the notion of {\em expected separable overapproximation (ESO)} \cite{PCDM, Pegasos2, SPCDM, DQA, Hydra, APPROX, Hydra2, NSync, QUARTZ} and \cite{Paper1}. This refers to an inequality involving the objective function and the sampling (a random set valued mapping describing the law with which subsets of coordinates are selected at each iteration), capturing in a compact way certain smoothness properties of the function in a random subspace spanned by the sampled coordinates. 


A (coordinate) sampling $\hat{S}$ is a random set-valued mapping with values being subsets of $[n]\eqdef \{1,2,\dots,n\}$. It will be useful to write
\begin{equation}\label{eq:p_i}
p_i \eqdef \Prob(i\in \hat{S}), \quad i\in [n].
\end{equation}

\begin{defn}[Expected Separable Overapproximation]
Let $f: \R^n\to \R$  be a differentiable function and $\hat{S}$ a sampling. We say that $f$ admits an expected separable overapproximation (ESO) with respect to sampling $\hat{S}$ with parameters $v=(v_1,\dots,v_n)>0$ if the following inequality holds\footnote{This definition can in a straightforward way be extended the case when coordinates are replaced by {\em blocks} of coordinates \cite{PCDM}. In such a case, $h_i$ would be a allowed to be a vector of size larger than one, $e_i$ would be replaced by a column submatrix of the identity matrix (usually denoted $U_i$ i n the literature) and $h_i^2$ would be replaced by the squared norm of $h_i$ (it is often useful to design this norm based on properties of $f$).} for all $x,h\in \R^n$:

\begin{equation}\label{eq:ESOmain_def}\Exp\left[ f\left(x + \sum_{i\in \hat{S}} h_i e_i \right) \right] \leq f(x) +\sum_{i=1}^n p_i (\nabla_i f(x))^{\top} h_i + \frac{1}{2} \sum_{i=1}^n p_i v_i h_i^2.\end{equation}

We will compactly write $(f,\hat{S})\sim ESO(v)$.
\end{defn}

In this definition, $e_i$ is the $i$-th unit coordinate vector in $\R^n$ and $\nabla_i f(x)=(\nabla f(x))^\top e_i$ is the $i$-th partial derivative of $f$ at $x$. In the context of {\em block} coordinate descent, the above definition refers to the case when all blocks correspond to coordinates. For simplicity of exposition, we focus on this case. However, all our  results can be extended to the more general block setup.

Instead of the above general definition, it will be useful to the reader to instead think about the form of this inequality in the simple case when $f(x)=\|Ax\|^2$, where $\|\cdot\|$ is the L2 norm,  and $x=0$. Letting $A=[A_1,\dots,A_n]$, in this case inequality \eqref{eq:ESOmain_def} takes the form
\[\Exp\left[ \left\|\sum_{i\in \hat{S}} A_i h_i \right\|^2\right] \leq h^\top\Diag(p\circ v) h,\]
where $p\circ v$ denotes the Hadamard product of vectors $p=(p_1,\dots,p_n)$ and $v=(v_1,\dots,v_n)$; that is $p\circ v = (p_1v_1,\dots,p_nv_n)\in \R^n$, and $\Diag(p\circ v)$ is the $n$-by-$n$ diagonal matrix with vector $p\circ v$ on the diagonal. The term on the left hand side is a convex quadratic function of $h$, and so is the term on the right hand side -- however, the latter function has a diagonal Hessian. Hence,  for quadratics, finding the ESO parameter $v$ reduces to an eigenvalue problem.

The ESO inequality is of key  importance for randomized coordinate descent methods for several reasons:

\begin{itemize}
\item The parameters $v=(v_1,\dots,v_n)$ for which ESO holds are needed\footnote{All existing {\em parallel} coordinate coordinate descent methods for which a complexity analysis has been performed are designed with fixed stepsizes. Designing a line-search procedure in such a setup is a nontrivial task, and to the best of our knowledge, only a single paper in the literature deals with this issue \cite{UNI}. Certainly, properly designed line search has the potential to improve the practical performance of these methods.} to run coordinate descent. Indeed, they are used to set the stepsizes to a suitable value.  

\item The size of these parameters directly influences the complexity of the method (see Table~\ref{tbl:complexity}). 

\item There are problems for which updating more coordinates in each iteration, as opposed to updating just one, may {\em not} lead to fewer iterations \cite{PCDM} (which suggests that perhaps the resources should be instead utilized in some other way). Whether this happens or not can be understood through a careful study of the complexity result and its dependence, through the vectors $p$ and $v$, on the number of coordinates updated in each iteration \cite{PCDM, Pegasos2, Hydra, Hydra2, SPCDM, QUARTZ}. 

\item The ESO assumption is {\em generic} in the sense that as soon as function $f$ and sampling $\hat{S}$ satisfy it, the complexity result follows. This leads to a natural dichotomy  in the study of coordinate descent: i) the search for new variants of coordinate descent (e.g., parallel, accelerated, distributed) and study of their complexity under the  ESO assumption, and ii) the search for pairs  $(f,\hat{S})$ for which one can compute $v$ such that $(f,\hat{S})\sim ESO(v)$. Our current study follows this dichotomy: in \cite{Paper1} we deal with the algorithmic and complexity aspects, and in  this paper we deal with the ESO aspect.

\end{itemize}

\subsection{Complexity of coordinate descent}  As mentioned above, complexity of coordinate descent methods  depends in a crucial way on the optimization problem, sampling employed, and on the ESO parameters $v=(v_1,\dots,v_n)$. In Table~\ref{tbl:complexity} we summarize all known complexity results\footnote{We exclude from the table some earlier results \cite{UCDC}, where an arbitrary {\em serial} sampling was analyzed; i.e., sampling $\hat{S}$ for which $\Prob(|\hat{S}|=1)=1$. The situation is much simpler for serial samplings.} which hold for an {\em arbitrary sampling}. Note that in all cases, vectors $p$ and $v$ appear in the complexity bound. The bounds are not directly comparable as they apply to different optimization problems.

\begin{table}[!t]
\centering
\begin{tabular}{|c|c|c|}
\hline
Setup & Complexity & Method / Paper / Year \\
\hline
&&\\
\begin{tabular}{c}Strongly convex\\ Smooth
\end{tabular} & $\displaystyle \max_i \left(\frac{v_i}{ p_i \lambda}\right) \times \log\left(\frac{1}{\epsilon}\right)$ & NSync \cite{NSync}, 10/2013\\
&&\\
\begin{tabular}{c} Strongly convex\\
nonsmooth\\ (Primal-dual)
\end{tabular} & $\displaystyle \max_i \left( \frac{1}{p_i} + \frac{v_i}{p_i \lambda n}\right) \times \log \left(\frac{1}{\epsilon}\right)$ & QUARTZ \cite{QUARTZ}, 11/2014\\
&&\\
\begin{tabular}{c}Convex \\smooth 
\end{tabular}& $\displaystyle \sqrt{2\sum_{i=1}^n \frac{v_i (x^0_i-x^*_i)^2}{p_i^2}} \times \frac{1}{\sqrt{\epsilon}}$ & ALPHA \cite{Paper1}, 12/2014\\
&&\\
\hline
\end{tabular}
\caption{Complexity of randomized coordinate descent methods which were analyzed for an arbitrary sampling ($\lambda$ is a strong convexity constant, $x_0$ is the starting point and $x_*$ the optimal point. }
\label{tbl:complexity}
\end{table}

For instance, the NSync bound\footnote{Complexity of NSync  depends on the initial ($x_0$) and optimal ($x_*$) points,  but have hidden this dependence. The full bound is obtained by replacing $\log(1/\epsilon)$ by $\log((f(x_0)-f(x_*))/\epsilon)$, where $f$ is the objective function. } in Table~\ref{tbl:complexity} applies to the problem of unconstrained minimization of a smooth strongly convex function. It was in  \cite{NSync} where the general form  of the ESO inequality used in this paper was first  mentioned and used to derive a complexity result for a coordinate descent method with arbitrary sampling. 

The Quartz algorithm \cite{QUARTZ}, on the other hand, applies to a much more serious problem -- a problem of key importance  in machine learning. In particular, it applies to the regularized {\em empirical risk minimization} problem, where the loss functions are convex and have Lipschitz gradients and the regularizer is strongly convex and possibly nonsmooth. Coordinate ascent is applied to the dual of this problem, and the bound appearing in Table~\ref{tbl:complexity} applies to the duality gap\footnote{Complexity of Quartz depends on an initial pair of primal and dual vectors; we have omitted  this dependence from the table. The full complexity result is obtained by replacing $\log(1/\epsilon)$ by $\log(\Delta_0/\epsilon)$, where $\Delta_0$ is the difference between the primal and dual function values for a pair of  (primal and dual) starting points.}.

The APPROX method was first proposed in \cite{APPROX} and then generalized to an arbitrary sampling (among other things) in \cite{Paper1}. In its accelerated variant it enjoys a $O(1/\sqrt{\epsilon})$ rate, whereas it's non-accelerated variant has a slower $O(1/\epsilon)$ rate. Again, the complexity of the method explicitly depends on the vector of probabilities $p$ and the ESO parameter $v$. 

\subsection{Historical remarks} The ESO relation \eqref{eq:ESOmain_def} was first introduced  by Richt\'{a}rik and Tak\'{a}\v{c} \cite{PCDM} in the special case of {\em uniform samplings}, i.e.,  samplings for which $\Prob(i\in \hat{S})=\Prob(j\in \hat{S})$ for all coordinates $i,j\in \{1,2,\dots,n\}$. The uniformity condition is satisfied for a large variety of samplings, we refer the reader to \cite{PCDM} for a basic classification of uniform samplings (including overlapping, non-overlapping, doubly uniform, binomial, nice, serial and parallel samplings)  and to \cite{Hydra,Hydra2,QUARTZ} for further examples (e.g.,  ``distributed sampling''). The study of non-uniform samplings has until recently been confined to {\em serial} sampling only, i.e., to samplings which only pick a single coordinate at a time. In \cite{NSync} the authors propose a particular example of a {\em parallel nonuniform} sampling, where ``parallel'' refers to samplings for which $\Prob(|\hat{S}|>1)>0$, and ``non-uniform'' simply means not uniform. Further, they derive an ESO inequality for their sampling and a partially separable function. The proposed sampling is easy to generate (note that in  general a sampling is described by assigning distinct probabilities to all $2^n$ subsets of $[n]$, and hence most samplings will necessarily be hard to generate), and leads to strong ESO bounds which predict nearly linear speedup for NSync for sparse problems. A further example of a non-uniform sampling was given in \cite{QUARTZ}---the so-called ``product sampling''---and an associated ESO inequality derived. Intuitively speaking, this sampling samples sets of ``independent'' coordinates, which  leads to  complexity scaling linearly with the size of the sampled sets. To the best of our knowledge, this is the state of the art -- no further non-uniform samplings were proposed nor associated ESO inequalities derived.

\subsection{Contributions}
We now briefly list the contributions of this work.

\begin{enumerate}
\item  ESO inequalities were  previously established for special classes of samplings only, almost invariably for {\em uniform samplings} \cite{PCDM, Hydra,SPCDM, Hydra2, APPROX},  and often using seemingly disparate   approaches.  We give the first {\em systematic study of ESO inequalities} for {\em arbitrary} samplings. 

\item We recover existing ESO results by applying our general technique.

\item Our approach to deriving ESO inequalities is via the study of random principal submatrices of a positive semidefinite matrix. In particular, we give bounds on the largest eigenvalue  of the mean of the random submatrix. This may be of independent interest.
\end{enumerate}


\subsection{Outline of the paper}

Our paper is organized as follows. In Section~\ref{sec:functions} we  describe the class of functions ($f$) we consider in this paper and briefly establish some basic terminology related to samplings ($\hat{S}$).  In Section~\ref{sec:prob} we study probability matrices associated with samplings ($\bP(\hat{S})$), in Section~\ref{sec:eig} we study eigenvalues of these probability matrices ($\lambda(\bP(\hat{S}))$ and $\lambda'(\bP(\hat{S}))$) and in Section~\ref{sec-ESO} we design a general technique for computing parameter  $v=(v_1,\dots,v_n)$ for which the ESO inequality holds (i.e., for which $(f,\hat{S})\sim ESO(v)$). We illustrate the use of these techniques in Section~\ref{sec:applications} and
conclude with Section~\ref{sec:conclusion}.

\section{Functions and samplings} \label{sec:functions}

Recall that in the paper we are concerned with establishing inequality \eqref{eq:ESOmain_def} which we succinctly write as $(f,\hat{S})\sim ESO(v)$. In Section~\ref{subsec:functions} we describe the class of functions $f$ we consider in this paper and in Section~\ref{subsec:samplings} we briefly review several elementary facts related to samplings.

\subsection{Functions} \label{subsec:functions}

We assume in this paper that $f:\R^n\to \R$ is differentiable and that it satisfies the following assumption (however, the first time we will again talk about functions is in Section~\ref{sec-ESO}).

\begin{assump}\label{ass:f}
There is an $m$-by-$n$ matrix $\bA$ such that for all $x,h\in \R^n$,
\begin{equation} \label{eq:shs7hs8} f(x+h) \leq f(x) + \ve{\nabla f(x)}{ h } + \frac{1}{2}h^{\top} \bA^{\top} \bA h .\end{equation}
\end{assump}

In the subsequent text, we shall often refer to the set of columns of $\bA$ for which  the entry in the $j$-th row of $\bA$ is nonzero:
\begin{equation}\label{a-Jj}
J_j \eqdef \{i \in [n]\;:\; \bA_{ji}\neq 0\}.
\end{equation}

Assumption~\ref{ass:f} holds for many functions of interest in optimization and machine learning. Coordinate  descent methods for functions $f$ explicitly required to satisfy Assumption~\ref{ass:f} were studied in \cite{SHOTGUN, Hydra, Hydra2}. 

The following simple observation will help us relate the above assumption with standing assumptions  considered in various papers on randomized coordinate descent methods.

\begin{prop}\label{lem:sum_structure_1}
Assume $f$ is of the  form 

\begin{equation}\label{eq:sumform}
f(x)=\sum_{j=1}^s  \phi_j(\bM_j x),
  \end{equation} 
  where for each $j$,  $\bM_j \in \R^{d\times n}$ and
function $\phi_j:\R^d \to \R$ has
$\gamma_j$-Lipschitz continuous gradient (with respect to the L2 norm).  Then $f$ satisfies Assumption~\ref{ass:f} for matrix $\bA$ given by
\[\bA^\top \bA = \sum_{j=1}^s \gamma_j \bM_j^\top \bM_j.\]
\end{prop}

\begin{proof} Pick $x,h\in \R^n$ and let $f_j(x)\eqdef \phi_j(\bM_j x)$. Then
 since $\phi_j$ is $\gamma_j$-smooth, we have
\begin{eqnarray*}
f_j(x+h) =  \phi_j\left(\bM_j x + \bM_j h\right) &\leq& \phi_j \left(\bM_j x \right) + \ve{\nabla \phi_j\left(\bM_j x \right)}{\bM_j h} + \tfrac{\gamma_j}{2}\left\|\bM_j h\right\|^2\\
& = &  f_j(x) +  \ve{\nabla f_j(x)}{h} + \tfrac{\gamma_j}{2} h^T \bM_j^\top \bM_j h.
\end{eqnarray*}
It remains to add these inequalities for $j=1,\dots,s$.
\end{proof}

By $\bI$ we denote the $n$-by-$n$ identity matrix and for $S\subseteq [n]$ we will use the notation $\bI_{[S]}$ for the $n$-by-$n$ matrix obtained from $\bI$ by retaining  elements $\bI_{ii}$ for which $i\in S$  and zeroing out all other elements.

We now  apply Proposition~\ref{lem:sum_structure_1} to several special cases:
\begin{enumerate}
\item \textbf{Partial separability.} Let $d=n$ and $\bM_j = \bI_{[C_j]}$, where for each $j$, $C_j\subseteq [n]$. Then $f$ is of the form
\begin{equation}\label{eq:sbi8s98hs}f(x) = \sum_{j=1}^s \phi_j(\bI_{[C_j]} x).\end{equation}
That is, $\phi_j$ depends on coordinates of $x$ belonging to set $C_j$ only. By Proposition~\ref{lem:sum_structure_1}, $f$ satisfies \eqref{eq:shs7hs8}, 
where $\bA$ is the $n$-by-$n$ diagonal matrix given by \[\bA_{ii} = \sqrt{\sum_{j:i \in C_j} \gamma_j}, \qquad i\in [n].\]

Functions of the form \eqref{eq:sbi8s98hs} (i.e., partially separable functions) were considered in the context of {\em parallel coordinate descent methods} in \cite{PCDM}. However, in \cite{PCDM} the authors only assume the {\em sum}  $f$ to have a Lipschitz gradient (which is more general, but somewhat complicates the analysis), whereas we assume that all component functions $\{\phi_j\}_{j}$ have Lipschitz gradient.

\item \textbf{Linear transformation  of variables.} Let $s=1$. Then $f$ is of the form
\begin{equation} \label{eq:igs87t9uohsohs98}f(x) = \phi_1(\bM_1 x).\end{equation}
 By Proposition~\ref{lem:sum_structure_1}, $f$ satisfies \eqref{eq:shs7hs8}, 
where $\bA$ is  given by \[\bA = \sqrt{\gamma_1}\bM_1.\]

A functions of the form \eqref{eq:igs87t9uohsohs98} appears in the {\em dual problem} of the standard primal-dual formulation to which  {\em stochastic dual coordinate ascent} methods  are applied \cite{SDCA, MinibatchASDCA, IProx-SDCA, APCG, QUARTZ}.

\item \textbf{Sum of scalar functions depending on $x$ through an inner product.} Let $d=1$ and $\bM_j= e_j^T \bM$, where $\bM \in \R^{m\times n}$ and $e_j$ is the $j$-th unit coordinate vector in $\R^m$. Then $f$ is of the form \begin{equation}\label{eq:sum_hydra}
f(x)=\sum_{j=1}^m  \phi_j(e_j^\top \bM x).
  \end{equation} 
  By Proposition~\ref{lem:sum_structure_1}, $f$ satisfies \eqref{eq:shs7hs8}, with $\bA$ given by 
  \[\bA = \Diag(\sqrt{\gamma_1},\dots,\sqrt{\gamma_m})\bM.\]
Functions of the form \eqref{eq:sum_hydra} play an important role in the design of {\em efficiently implementable accelerated coordinate descent methods} \cite{APPROX, Paper1}. These functions also appear in the  
{\em primal problem} of the standard primal-dual formulation to which  {\em stochastic dual coordinate ascent} methods  are applied.
\end{enumerate}


\subsection{Samplings} \label{subsec:samplings}

As defined in the introduction, by sampling we mean a random set-valued mapping with values in $2^{[n]}$ (the set of subsets of $[n]$).

\paragraph{Classification of samplings.}  Following the terminology  established in \cite{PCDM}, we say that sampling $\hat{S}$ is {\em proper} if $p_i =\Prob(i\in \hat{S})>0$ for all $i\in [n]$. We shall focus our attention on proper samplings as otherwise there is a coordinate which is never chosen (and hence never updated by the coordinate descent method).  We say that $\hat{S}$ is {\em nil} if $\Prob(\hat{S}=\emptyset)=1$.

Of key importance in this paper are {\em elementary samplings}, defined next.

\begin{defn}[Elementary samplings]\label{df:elementary}
Elementary sampling associated with $S \subseteq [n]$ is sampling which selects set  $S$ with probability one. We will denote it by $\hat{E}_S$: $\Prob(\hat{E}_S=S) = 1$.
\end{defn}

By {\em image} of sampling $\hat{S}$ we mean the collection of sets which are chosen with positive probability: $\Image(\hat{S}) = \{S\subseteq [n]\;:\; \Prob(\hat{S}=S)>0\}$. We say that $\hat{S}$ is {\em nonoverlapping}, if no two sets in its image intersect.  We say that the sampling  is {\em uniform} if $\Prob(i\in \hat{S})=\Prob(j\in \hat{S})$ for all $i,j\in [n]$. The class of uniform samplings is large, for examples (and properties) of notable subclasses, we refer the reader to \cite{PCDM} and \cite{Hydra}. 

We say that sampling $\hat{S}$ is {\em doubly uniform} if it satisfies the following condition: if $|S_1|=|S_2|$, then $\Prob(\hat{S}=S_1)=\Prob(\hat{S}=S_2)$. Necessarily, every doubly uniform sampling is uniform \cite{PCDM}. The definition postulates an additional ``uniformity'' property (``equal cardinality implies equal probability''), whence the name. As described in \cite{PCDM}, doubly uniform samplings are special in the sense that ``good'' ESO results can be proved for them. A notable subclass of the class of doubly uniform samplings are the $\tau$-nice samplings for $1\leq \tau\leq n$.  The $\tau$-nice sampling is obtained by picking (all) subsets of cardinality $\tau$, uniformly at random (we give a precise definition below).
This sampling is by far the most common in stochastic optimization, and refers to standard mini-batching. The $\tau$-nice sampling arises as a special case of  the $(c,\tau)$-distributed sampling (which, as its name suggests, can be used to design distributed variants of coordinate descent \cite{Hydra,Hydra2}), which we define next:

\begin{defn}[$(c,\tau)$-distributed sampling; \cite{Hydra, Hydra2, QUARTZ}]
\label{defn:distrib} Let ${\cal P}_1, \dots, {\cal P}_c$ be a partition of $\{1,2,\dots,n\}$ such that $|{\cal P}_l| = s$ for all $l$. That is, $sc = n$. Now let $\hat{S}_1,\dots,\hat{S}_c$ be independent $\tau$-nice samplings from ${\cal P}_1, \dots, {\cal P}_c$, respectively.  Then the sampling 
\begin{equation} \label{eq:sjsiuhuis}\hat{S}\eqdef \bigcup_{l=1}^c \hat{S}_l,\end{equation}
is called {$(c,\tau)$-distributed sampling.}
\end{defn}

The $\tau$-nice sampling arises as a special case of the $(c,\tau)$-distributed sampling (for $c=1$) which we define next.

\begin{defn}[$\tau$-nice sampling; \cite{PCDM, Pegasos2, DQA, SPCDM, APPROX}]\label{def:tau-nice}
Sampling $\hat S$ is called $\tau$-nice if it picks only subsets of $[n]$ of cardinality $\tau$, uniformly at random. More formally, it is defined by 
\begin{equation}\Prob(\hat{S}=S) = \begin{cases} 1/{n \choose \tau}, & |S|=\tau,\\
0, & \text{otherwise}. \end{cases}\end{equation}
\end{defn}

\paragraph{Operations with samplings.} We now define several basic operations with samplings (convex combination, intersection and restriction).

\begin{defn}[Convex combination of samplings; \cite{PCDM}] Let $\hat{S}_1,\dots,\hat{S}_k$ be samplings and let $q_1,\dots,q_k$ be nonnegative scalars summing  to 1. By $\sum_{t=1}^k q_t \hat{S}_t$ we denote the sampling obtained as follows: we first pick $t\in \{1,\dots,k\}$, with probability $q_t$, and then  sample according to $\hat{S}_t$. More formally, $\hat{S}$ is defined as follows:
\begin{equation}\label{eq:convex_comb}\Prob(\hat{S}=S) = \sum_{t=1}^k q_t \Prob(\hat{S}_t = S), \qquad S\subseteq [n].\end{equation}
\end{defn}

Note that \eqref{eq:convex_comb} indeed defines  a sampling, since 
\[\sum_{S\subseteq [n]} \Prob(\hat{S}=S) = \sum_{S\subseteq [n]} \sum_{t=1}^k q_t \Prob(\hat{S}_t = S) =   \sum_{t=1}^k q_t  \sum_{S\subseteq [n]} \Prob(\hat{S}_t = S) =  \sum_{t=1}^k q_t = 1.\]

Each sampling is a convex combination of elementary samplings. Indeed, for each $\hat{S}$ we have
\begin{equation}\label{eq:9hs8h9h}\hat{S} = \sum_{S\subseteq [n]} \Prob(\hat{S}=S) \hat{E}_S.\end{equation}

We now show that each doubly uniform sampling arises as a convex combination of $\tau$-nice samplings.

\begin{prop} \label{prop:DU}Let $\hat{S}$ be a doubly uniform sampling and let $\hat{S}_\tau$ be the $\tau$-nice sampling, for $\tau=0,1,\dots,n$. Then
\[\hat{S} = \sum_{\tau=0}^n \Prob(|\hat{S}|=\tau) \hat{S}_\tau.\]
\end{prop}
\begin{proof} Fix any $S\subseteq [n]$ and let $q_\tau=\Prob(|\hat{S}|=\tau)$. Note that
\[\Prob(\hat{S}=S) =  \sum_{\tau=0}^n \Prob(\hat{S}=S \; \& \; |\hat{S}|=\tau ) = \sum_{\tau=0}^n q_\tau \Prob(\hat{S}=S \;|\; |\hat{S}|=\tau) =  
\sum_{\tau=0}^n q_\tau \Prob(\hat{S}_\tau =S),
\]
where the last equality follows from the definition of doubly uniform and $\tau$-nice  samplings. The statement then follows from  \eqref{eq:convex_comb} (i.e., by definition of convex combination of samplings).
\end{proof}

It will be useful to define two more operations with samplings; intersection and restriction.

\begin{defn}[Intersection of samplings]\label{def:intersection}
For two samplings  $\hat{S}_1$ and $\hat{S}_2$ we define the intersection $\hat{S}\eqdef \hat{S}_1\cap \hat{S}_2$ as the sampling for which:
\[\Prob(\hat{S}=S) = \Prob(\hat{S}_1\cap \hat{S}_2=S), \quad S\subseteq [n].\]
\end{defn}

\begin{defn}[Restriction of a sampling]\label{def:restriction}
Let $\hat{S}$ be a sampling and $J \subseteq [n]$. By restriction of $\hat{S}$ to $J$ we mean the sampling $\hat{E}_J\cap \hat{S}$.
By abuse of notation we will also write this sampling as $J\cap  \hat S$.
\end{defn}

\paragraph{Graph sampling.} Let $G=(V,E)$ be an undirected graph with $|V|=n$ vertex and 
$(i,i')$ be an edge in $E$ if and only if there is $j \in[m]$ such that $\{i,i'\}\subseteq J_j $.  If $S$ is an independent set of graph $G$, then necessarily
$$\max_{j\in [m]} | J_j \cap S| = 1.$$
Denote by $\cT$ the collection of all independent sets of the graph $G$. We now define the graph sampling as follows:

\begin{defn}[Graph sampling]
Graph sampling associated with graph $G$  is any sampling $\hat{S}$ for which $\Prob(\hat{S}=S)=0$ if $S\notin \cT$. In other words, a graph sampling can only assign positive weights to independent sets of $G$.
\end{defn}

Let $\hat{S}$ be a graph sampling. In view of \eqref{eq:9hs8h9h}, for some nonnegative constants $q_S$ adding up to 1:
$$
\hat S=\sum_{S\in \cT } q_S \hat E_S
$$
 Note that, necessarily,  $q_S=\Prob(\hat S=S)$ for all $S\in \cT$.

\begin{defn}[Product sampling] \label{def:prosam}
Let $X_1,\dots,X_{\tau}$ be a partition of $[n]$, i.e.,
$$
\displaystyle  X_1\cup \dots X_\tau=[n]; \enspace X_i\cap X_j=\emptyset,\enspace\forall 1\leq i<j\leq n.$$
Define:
$$
{\cal S} \eqdef  X_{1}\times \cdots \times X_{\tau}.
$$
The {product sampling}
 $\hat{S}$ is obtained by choosing $S \in {\cal S}$, uniformly at random; that is, via: \begin{equation}\label{eq:si9s8hs}\Prob(\hat{S}=S) = \frac{1}{|{\cal S}|} = \frac{1}{\prod_{l=1}^\tau |X_l|}, \quad S \in {\cal S}. \end{equation}
\end{defn}
A similar sampling was first considered in~\cite[Section 3.3]{QUARTZ} with an additional \textit{group separability assumption} on the partition
$X_1,\dots,X_{\tau}$, which can be equivalently 
stated as:
$$
\max_{j\in m} |J_j\cap S|=1,\enspace \forall S\in {\cal S}.
$$
In other words, it is both a product sampling and graph sampling.
Note that in Definition~\ref{def:prosam} we do not make any assumption on the partition.
Also, the product sampling is a nonuniform sampling as long as all the 
sets $X_l$ do not have the same cardinality, which occurs necessarily if $\tau$, representing the number of processors, is not divisible by $n$.

\section{Probability matrix associated with a sampling}\label{sec:prob}

In this section we define the notion of a {\em probability matrix} associated with a sampling.  As we shall see in later sections, this matrix encodes all information about $\hat{S}$ which is relevant for development of ESO inequality.

\begin{defn}[Probability matrix]
With each  sampling $\hat{S}$ we associate 	an  $n$-by-$n$ ``probability matrix'' $\bP=\bP(\hat{S})$ defined by
\[\bP_{ij} = \Prob(\{i,j\} \subseteq  \hat{S} ), \qquad  i,j\in[n].\]
\end{defn}

We shall write $\bP(\hat{S})$ when it is important to indicate which sampling is behind the probability matrix, otherwise we simply write $\bP$.

For two  matrices  $\bM_1$ and $\bM_2$ of the same size, we denote by $\bM_1\circ \bM_2$ their Hadamard (i.e., elementwise) product. We use the same notation for Hadamard product of vectors. For arbitrary matrix $\bM\in \R^{n\times n}$ and $S\subseteq [n]$ we will use the notation $\bM_{[S]}$ for the $n$-by-$n$ matrix obtained from $\bM$ by retaining  elements $\bM_{ij}$ for which both $i\in S$ and $j\in S$ and zeroing out all other elements.
In what follows,  by $\bE$ we denote the $n$-by-$n$ matrix of all ones and by $\bI$ we denote the $n$-by-$n$ identity matrix. For  any $h=(h_1,\dots,h_n)\in \R^n$ and $S\subseteq [n]$ we will write
\begin{equation}\label{eq:ks43947}
h_{[S]} \eqdef \sum_{i\in S} h_i e_i = \bI_{[S]} h,
\end{equation}
where $e_1,\dots,e_n$ are the standard basis vectors in $\R^n$. Also note that 
\begin{equation}\label{eq:sig98s445566}\bM_{[S]} = \bE_{[S]}\circ \bM = \bI_{[S]} \bM \bI_{[S]}.\end{equation}

Using the notation we have just established, probability matrices of  elementary samplings  are given by
\begin{equation}\label{eq:bPEJ}\bP(\hat{E}_S) = \bE_{[S]} = e_{[S]} e_{[S]}^\top,\end{equation}
where $e\in \R^n$ is the  vector of all ones. In particular, the matrix is  rank-one and positive semidefinite. 


\subsection{Representation of probability matrices}

We now establish a simple but particularly insightful result, leading to many useful identities.

\begin{theo}\label{thm:8sy98ys}For each sampling $\hat{S}$ we have
\begin{equation}\label{eq:s989shsos}\bP(\hat{S}) = \Exp \left[\bE_{[\hat{S}]}\right] = \sum_{S\subseteq [n]}\Prob(\hat{S}=S) \bE_{[S]}.\end{equation}
In particular:
\begin{itemize}
\item[(i)] The set of probability matrices is the convex hull of the probability matrices corresponding to elementary samplings.
\item[(ii)] $\bP(\hat{S})\succeq 0$ for each $\hat{S}$.
\end{itemize}
\end{theo}
\begin{proof} The $(i,j)$ element of the matrix on the right hand side is $\Exp \left[(\bE_{[\hat{S}]})_{ij}\right]$. Since $(\bE_{[\hat{S}]})_{ij} = 1$ if $\{i,j\}\subseteq \hat{S}$ and $(\bE_{[\hat{S}]})_{ij} = 0$ otherwise, we have  $\Exp \left[(\bE_{[\hat{S}]})_{ij}\right] = \Prob(\{i,j\}\subseteq \hat{S}) =  (\bP(\hat{S}))_{ij}$.  Claim (i) follows from \eqref{eq:s989shsos} since $\bE_{[S]}=\bP(\hat{E}_S)$. Claim (ii) follows from \eqref{eq:s989shsos} since $\bE_{[S]}\succeq 0$ for all $S\subseteq [n]$.
\end{proof}

We have the following useful corollary:\footnote{Identities \eqref{eq:jsu8js8}--\eqref{eq:is8js8sos0trace} were already established in \cite{PCDM}, in a different way  without relying on Theorem~\ref{thm:8sy98ys}, which is new. However, in this paper a key role is played by identities \eqref{eq:sj78jdi}--\eqref{eq:sjdud7xxx}, which are also new. It was while proving these identities that we realized the fundamental nature of Theorem~\ref{thm:8sy98ys}, as a vehicle for obtaining all identities in Corollary~\ref{lem:main} as a  consequence. The identities will be needed in further development. For illustration of a different proof technique, here is an alternative proof of \eqref{eq:sjdud7xxx}: \[\textstyle
\Exp\left[ h_{[\hat{S}]}^T \bM h_{[\hat{S}]} \right]  =
 \Exp\left[ \sum_{(i,i')\in \hat{S}\times \hat{S}} \bM_{ii'}h_{i'} h_i\right] 
= \sum_{(i,i')\in [n]\times [n]}\Prob(i\in \hat{S},i'\in \hat{S})\bM_{ii'}h_{i'} h_i \;=\; h^T(\bP\circ \bM)h\].}

\begin{coro}
\label{lem:main}
 Let $\hat{S}$ be any sampling, $\bP=\bP(\hat{S})$, $\bM \in \R^{n \times n}$ be an arbitrary matrix and $h\in \R^n$. Then the following identities hold:
\begin{eqnarray}
 \label{eq:sj78jdi} 
 \bP \circ \bM  &=& \Exp\left[ \bM_{[\hat{S}]}\right]\\
\label{eq:sjdud7xxx}
h^\top \left( \bP \circ \bM\right) h &=& \Exp\left[ h^\top \bM_{[\hat{S}]} h\right] \;=\; \Exp\left[ h_{[\hat{S}]}^\top \bM h_{[\hat{S}]} \right] \\ 
\label{eq:jsu8js8}
h^\top \bP h &=&  \Exp \left[ \Big(\sum_{i \in \hat{S}} h_i\Big)^2\right] \\
\label{eq:js9s0sj}
\sum_{i=1}^n \bP_{ii} h_i   &=&   \Exp \left[ \sum_{i \in \hat{S}} h_i\right]\\
\label{eq:is8js8sos0}
e^T \bP e &=& \Exp \left[  |\hat{S}|^2 \right]\\
\label{eq:is8js8sos0trace}
\Tr{(\bP)} &=& \Exp \left[  |\hat{S}| \right]
\end{eqnarray}
\end{coro}
\begin{proof}
Since multiplying  a matrix in the Hadamard sense by a fixed matrix  is a linear operation,
\[\bP\circ \bM \;\; \overset{\eqref{eq:s989shsos}}{=} \;\; \Exp\left[\bE_{[\hat{S}]}\right] \circ \bM \;\;=\;\; \Exp\left[\bE_{[\hat{S}]}\circ  \bM\right]  \;\;\overset{\eqref{eq:sig98s445566}}{=}\;\; \Exp \left[\bM_{[\hat{S}]}\right].\] 
Next, identity \eqref{eq:sjdud7xxx} follows from \eqref{eq:sj78jdi}:
\begin{eqnarray*}
h^\top (\bP \circ \bM) h &=& h^\top \Exp \left[\bM_{[\hat{S}]}\right] h \;\;
=\;\; \Exp   \left[ h^\top \bM_{[\hat{S}]}h \right] \;\; \overset{\eqref{eq:sig98s445566}}{=}  \;\; \Exp   \left[ h^\top \bI_{[\hat{S}]} \bM \bI_{[\hat{S}]} h \right] 
\overset{\eqref{eq:ks43947}}{=}  \Exp\left[h_{[\hat{S}]}^\top \bM h_{[\hat{S}]}\right].
\end{eqnarray*}
Identity \eqref{eq:jsu8js8} follows from \eqref{eq:sjdud7xxx} by setting $\bM=\bE$: 
\[\Exp\left[h_{[\hat{S}]}^\top \bE h_{[\hat{S}]}\right] = \Exp\left[ \sum_{i,j\in \hat{S}} h_i h_j \right] = \Exp \left[ \Big(\sum_{i \in \hat{S}} h_i\Big)^2\right].\]
Identity \eqref{eq:js9s0sj} holds since
\[\sum_i \bP_{ii} h_i =\sum_i \sum_{S: i\in S}\Prob(\hat{S}=S) h_i  = \sum_{S\subseteq [n]} \Prob(\hat{S}=S)  \sum_{i\in S} h_i = \Exp \left[ \sum_{i\in \hat{S}} h_i\right].\]
Finally, \eqref{eq:is8js8sos0}  (resp.  \eqref{eq:is8js8sos0trace}) follows from  \eqref{eq:jsu8js8} (resp. \eqref{eq:js9s0sj}) by setting $h=e$.
\end{proof}

If $\hat{S}$ is a uniform sampling (i.e., if $\Prob(i\in \hat{S})=\Prob(j\in \hat{S})$ for all $i,j\in [n]$), then from \eqref{eq:is8js8sos0trace} we deduce that for all $i\in [n]$:
\begin{equation}\label{eq:uniform_xx} p_i \equiv \Prob(i\in \hat{S}) \equiv \bP_{ii} = \frac{\Exp [|\hat{S}|]}{n}.\end{equation}

\subsection{Operations with samplings}

We now give formulae for the probability matrix of the sampling arising as a convex combination, intersection or a restriction, in terms of the probability matrices of the constituent samplings.

\paragraph{Convex combination of samplings.} We have seen in \eqref{eq:9hs8h9h} that each sampling is a convex combination of elementary samplings. 
In view of Theorem~\ref{thm:8sy98ys},  the probability matrices of the samplings  are related the same way:
\begin{equation}\label{eq:9hs8h9hxxx}\bP(\hat{S}) = \sum_{S\subseteq [n]} \Prob(\hat{S}=S) \bP(\hat{E}_S).\end{equation}

More generally, as formalized in the following lemma, the probability matrix of a convex combination of samplings is equal to the convex combination of the probability matrices of these samplings.

\begin{lemma} \label{lem:convex_comb} Let $\hat{S}_1, \dots, \hat{S}_k$ be samplings and $q_1,\dots,q_k$ be non-negative scalars summing up to 1.
Then 
\begin{equation}\label{eq:s98h9s8hspp}
\bP\left(\sum_{t=1}^k q_t \hat{S}_t\right) = \sum_{t=1}^k q_t \bP(\hat{S}_t).
\end{equation}
\end{lemma}
\begin{proof} Let $\hat{S}$ be the convex combination of samplings $\hat{S}_1,\dots,\hat{S}_k$ and fix any $i,j\in [n]$. By definition, \begin{eqnarray*}
(\bP(\hat{S}))_{ij} &=& \Prob(\{i,j\}\subseteq \hat{S}) \;\;=\;\; \sum_{S\subseteq [n] \;:\; \{i,j\} \subseteq S} \Prob(\hat{S}=S)\\ &\overset{\eqref{eq:convex_comb}}{=}&
\sum_{S\subseteq [n] \;:\; \{i,j\} \subseteq S} \sum_{t=1}^k q_t \Prob(\hat{S}_t = S) \;\;= \;\; \sum_{t=1}^k q_t \sum_{S\subseteq [n] \;:\; \{i,j\} \subseteq S}  \Prob(\hat{S}_t = S)\\
&=& \sum_{t=1}^k q_t \Prob(\{i,j\}\subseteq \hat{S}_t) \;\; =\;\; \sum_{t=1}^k q_t (\bP(\hat{S}_t))_{ij} \;\; = \;\; \left(\sum_{t=1}^k q_t \bP(\hat{S}_t) \right)_{ij}.
\end{eqnarray*}
\end{proof}

\paragraph{Intersection of samplings.} The probability matrix of the intersection of two independent samplings is equal to the Hadamard product of the probability matrices of these samplings. This is formalized in the following lemma.

\begin{lemma}\label{lem:IIS} Let $\hat{S}_1, \hat{S}_2$ be independent samplings. Then 
\[\bP(\hat{S}_1 \cap \hat{S}_2) = \bP(\hat{S}_1)\circ \bP(\hat{S}_2).\]
\end{lemma}
\begin{proof}
$[\bP(\hat{S}_1\cap \hat{S}_2)]_{ij} = \Prob(\{i,j\}\subseteq \hat{S}_1\cap \hat{S}_2) = \Prob(\{i,j\} \subseteq \hat{S}_1) \Prob( \{i,j\}\subseteq \hat{S}_2 )= [\bP(\hat{S}_1)]_{ij} [\bP(\hat{S}_2)]_{ij}.
$
\end{proof}

\paragraph{Restriction.}

By Lemma~\ref{lem:IIS}, the probability matrix of the restriction of arbitrary sampling $\hat{S}$ to $J\subseteq [n]$ is given by (we give several alternative ways of writing the result):
\begin{equation}\label{eq:98shsguy8658}\bP(J\cap \hat{S}) = \bP(\hat{E}_J)\circ \bP(\hat{S}) \overset{\eqref{eq:bPEJ}}{=} \bE_{[J]}\circ \bP(\hat{S}) = \bI_{[J]} \bP(\hat{S}) \bI_{[J]}.\end{equation}
Note that $\bP(J\cap \hat{S})$ is the matrix obtained from $\bP(\hat{S})$ by keeping only elements $i,j\in J$ and zeroing out all the rest. Furthermore, by combining the formulae derived above, we get
\begin{equation}\label{eq:iuhsi655r6sft}\bP\left(J \cap \sum_{t=1}^k q_t \hat{S}_t\right) \overset{\eqref{eq:98shsguy8658}+\eqref{eq:s98h9s8hspp}}{=} \bE_{[J]} \circ \left(\sum_{t=1}^k q_t \bP(\hat{S}_t)\right) =  \sum_{t=1}^k q_t \left( \bE_{[J]} \circ \bP(\hat{S}_t)\right) \overset{\eqref{eq:98shsguy8658}}{=} \sum_{t=1}^k q_t \bP(J\cap \hat{S}_t).\end{equation}

\subsection{Probability matrix of special samplings}

The probability matrix of the $(c,\tau)$-distributed samplings is computed in the following lemma.

\begin{lemma}  \label{lem:su8dbd}
Let $\hat{S}$ be the ($c,\tau$)-distributed sampling associated with the partition $\{\pP_1,\dots,\pP_c\}$ of $[n]$ such that $s=|\pP_l|$ for $l\in [c]$ (see Definition~\ref{defn:distrib}). Then
\begin{equation}\label{eq:poiuyth778}
\bP(\hat{S}) = \frac{\tau}{s}\left[\alpha_1 \bI + \alpha_2 \bE + \alpha_3 (\bE-\bB)\right],
\end{equation}
where 
\[\alpha_1 = 1-\frac{\tau-1}{s_1}, \qquad \alpha_2 = \frac{\tau-1}{s_1}, \qquad \alpha_3 = \frac{\tau}{s} - \frac{\tau-1}{s_1},\]
$s_1=\max(s-1,1)$ and
\begin{equation}\label{eq:BBB}\bB = \sum_{l=1}^c \bP ({\hat E_{{\cal P}_l}}).\end{equation} Note that $\bB$ is the 0-1 matrix with $\bB_{ij}=1$ if and only if
 $i,j$ belong to the same partition.
\end{lemma}
\begin{proof}
 Let $\bP=\bP(\hat{S})$. It is easy to see that 
\begin{equation*}
 \bP_{ij}=\left\{\begin{array}{ll}
  \frac{\tau}{s} & \mathrm{if~~}  i=j\\
\frac{\tau(\tau-1)}{s s_1}  & \mathrm{if~~}  i\neq j \mathrm{~and~} i,j\in \pP_l \mathrm{~for~some~} l\in[c]\\
\frac{\tau^2}{s^2}  & \mathrm{otherwise~~}  \\
 \end{array}\right.
\end{equation*}
Hence,
\[\bP=\frac{\tau}{s}\bI+\frac{\tau(\tau-1)}{s s_1}(\bB-\bI)+\frac{\tau^2}{s^2}(\bE-\bB)=\frac{\tau}{s}
[\alpha_1 \bI+\alpha_2 \bE+\alpha_3(\bE-\bB)].\qedhere\]
\end{proof}

As a corollary of the above in the  $c=1$ case we obtain the probability matrix of the $\tau$-nice sampling:

\begin{lemma} \label{lem:prob_m_of_tau_nice}
Fix $1\leq \tau \leq n$ and let $\hat{S}$ be the $\tau$-nice sampling. Then
\begin{equation}\label{eq:jhs6789hs8s}\bP(\hat{S}) = \frac{\tau}{n} \left( (1- \beta) \bI + \beta \bE \right),\end{equation}
where $\beta = (\tau-1)/\max(n-1,1)$. If $\tau=0$, then $\bP(\hat{S})$ is the zero matrix.
\end{lemma}
\begin{proof} For $\tau\geq 1$  this follows from Lemma~\ref{lem:su8dbd} in the special case when $c=1$ (note that ${\cal P}_1=[n], s=n$ and $\bB=\bE$). 
\end{proof}

Finally, we compute the probability matrix of a doubly uniform sampling.

\begin{lemma} \label{lem:ius98ss}Let $\hat{S}$ be a doubly uniform sampling and assume it is not nil (i.e., assume that $\Prob(\hat{S}=\emptyset)
\neq 1$). Then
\begin{equation}\label{eq:du_prob_marix}
\bP(\hat{S}) = \frac{\Exp[|\hat{S}|]}{n} \left( (1 - \beta ) \bI + \beta\bE \right),
\end{equation}
where
\begin{equation}\label{eq:beta}\beta = \left(\tfrac{\Exp[|\hat{S}|^2]}{\Exp[|\hat{S}|]} - 1\right)/\max(n-1,1).\end{equation}
\end{lemma} 
\begin{proof} Letting $q_\tau = \Prob(|\hat{S}|=\tau)$, by Proposition~\ref{prop:DU} we can write
$\hat{S} = \sum_{\tau=0}^n q_\tau\hat{S}_\tau$, 
where $\hat{S}_\tau$ is the $\tau$-nice sampling. It only remains to combine Lemma~\ref{lem:convex_comb} and Lemma~\ref{lem:prob_m_of_tau_nice} and rearrange the result.
\end{proof}

Note that Lemma~\ref{lem:prob_m_of_tau_nice} is a special case of Lemma~\ref{lem:ius98ss} (covering the case  when $\Prob(|\hat{S}|=\tau)=1$ for some $\tau$).

\section{Largest eigenvalues of the probability matrix} \label{sec:eig}

For an $n\times n$ positive semidefinite matrix $\bM$ we denote by $\lambda(\bM)$ the largest eigenvalue of $\bM$:
\begin{equation}\label{eq:lambdaM} \lambda(\bM) \eqdef \max_{h\in \R^n} \{h^\top \bM h \;:\; h^\top  h \leq 1\}.\end{equation}
For a vector $v\in \R ^n$, let $\Diag(v)$ be the diagonal matrix with  $v$ on the diagonal. For an $n$-by-$n$ matrix $\bM$, $\Diag(\bM)$ denotes the diagonal matrix containing the diagonal of $\bM$.  By $\lambda'(\bM)$ we shall denote the ``normalized'' largest eigenvalue of $\bM$: 
\begin{equation} \label{eq:lambda'M}\lambda'(\bM) \eqdef \max_{h\in \R^n} \{h^\top \bM h \;:\; h^\top \Diag(\bM) h \leq 1\}. \end{equation}
Note that $1\leq \lambda'(\bM)\leq n$.

In this section we study (standard and normalized) largest eigenvalue of the probability matrix associated with a sampling:
\begin{equation}\label{eq:98s98soiuoiuxxx} \lambda(\hat{S}) \eqdef  \lambda(\bP(\hat{S})) \overset{\eqref{eq:lambdaM}}{=} \max_{h\in \R^n} \{h^\top \bP(\hat{S}) h \;:\; h^\top  h \leq 1\} \end{equation} and 
\begin{equation}\label{eq:98s98soiuoiu} \lambda'(\hat{S}) \eqdef \lambda'( \bP(\hat{S}) ) \overset{\eqref{eq:lambda'M}}{=} \max_{h\in \R^n} \{h^\top \bP(\hat{S}) h \;:\; h^\top \Diag(\bP(\hat{S})) h \leq 1\}. \end{equation} 
Recall that by Theorem~\ref{thm:8sy98ys}, $\bP(\hat{S})$ is positive semidefinite for each sampling $\hat{S}$.  For convenience, we write $\lambda(\hat{S})$ (resp. $\lambda'(\hat{S})$) instead of $\lambda(\bP(\hat{S}))$  (resp. $\lambda'(\bP(\hat{S}))$). We study these quantities since, as we will show in later sections, they are useful in computing parameter $v=(v_1,\dots,v_n)$ for which ESO  holds.  

If $\hat{S}$ is a uniform sampling  (i.e., if $\Prob(i\in \hat{S})=\Prob(j\in \hat{S})$ for all $i,j\in [n]$), then since $\Tr(\bP(\hat{S}))=\Exp[|\hat{S}|]$ (see \eqref{eq:uniform_xx}), we have $\Diag(\bP(\hat{S}))=\tfrac{\Exp[|\hat{S}|]}{n} \bI$, from which we obtain (assuming that $\hat{S}$ is not nil):
\begin{equation}\label{eq:uniform_eigs_009}
\lambda'(\hat{S}) = \frac{n}{\Exp[|\hat{S}|]}\lambda(\hat{S}).
\end{equation}

\subsection{Elementary samplings}

In the case of elementary samplings the situation is simple. Indeed, for any $J\subseteq [n]$, we have
\begin{align}\label{a-PEJ}
\lambda'(\hat{E}_J)  = \lambda(\hat E_J)\overset{\eqref{eq:bPEJ}}{=} \lambda(e_{[J]} e_{[J]}^\top) = e_{[J]}^\top e_{[J]} =|J|.
\end{align}
This can, in fact, be seen as a consequence of a more general identity\footnote{The proof is immediate: fixing $x$, for any $h\in \R^n$ we have $(h^\top x)^2 = ((h\circ x)^\top e )^2 = ((h\circ x)^\top e_{[S]} )^2$, where $e$ is the vector of all ones, $S=\{i\;:\; x_i\neq 0\}$ and the entries of $e_{[S]}$ are 1 for $i\in S$ and 0 otherwise. It only remains to apply the Cauchy-Schwartz inequality, which  is attained,  whence the identity.} for arbitrary symmetric rank one matrices: for any $x\in \R^n$, we have
\begin{equation}\label{eq:09s09shjyu}\lambda'(xx^\top) = \|x\|_{0} \eqdef  |\{i\;:\; x_i\neq 0\}|.\end{equation}
Since $\bP(\hat{E}_J) = \bE_{[J]}$ and $\Diag(\bE_{[J]})=\bI_{[J]}$,  \eqref{a-PEJ} can equivalently be written as
\begin{align}\label{a-PEJ-equiv}
\bE_{[J]} \preceq |J| \bI_{[J]},
\end{align}
and adding that the bound is tight.

\subsection{Bounds for arbitrary samplings}

In the first result of this section we give sharp bounds for $\lambda'(\hat{S})$ for arbitrary sampling $\hat{S}$.

\begin{theo} \label{th-lambda} Let $\hat{S}$ be an arbitrary sampling. 
\begin{enumerate}
\item[(i)] {\bf Lower bound.} If $\hat{S}$ is not nil (i.e., if $\Prob(\hat{S}\neq \emptyset)>0$), then
 \[1\leq \frac{\Exp[|\hat{S}|^2]}{\Exp[|\hat{S}|]} \leq \lambda'(\hat{S})
.\] 
\item[(ii)] {\bf Upper bound.} If $\tau$ is a constant such that $|\hat{S}|\leq \tau$ with probability 1, then
$\lambda'(\hat{S}) \leq \tau.$
\item[(iii)] {\bf Identity.} 
If $|\hat{S}|=\tau$ with probability 1, then
$\lambda'(\hat{S}) = \tau$.
\end{enumerate}
\end{theo}
\begin{proof}
\begin{enumerate}
\item[(i)] For simplicity, let $\bP = \bP(\hat{S})$. 
If $e\in \R^ n$ is the vector of all ones, then 
we get 
\[\lambda'(\hat{S}) \;\; \overset{\eqref{eq:98s98soiuoiu} }{\geq} \;\;
\frac{e^\top \bP e}{e^\top \Diag(\bP)e} \;\;=\;\; \frac{e^\top \bP e}{\Tr(\bP)} \;\; \geq \;\; 1,\]
where the last inequality holds since $\Tr(\bP)$ is upper bounded by the sum of all elements of $\bP$. It remains to apply identities
\eqref{eq:is8js8sos0} and \eqref{eq:is8js8sos0trace}.

\item[(ii)] In view of \eqref{eq:9hs8h9h}, we can represent $\hat{S}$ as a convex combination of elementary samplings:
$$\hat{S}=\sum_{S\subseteq [n]} q_S \hat{E}_S,$$ where $q_S=\Prob(\hat{S}=S)$. Since $|\hat{S}|\leq \tau$ with probability 1, we have $|S|\leq \tau$ whenever $q_S>0$.
Thus we  have 
\begin{equation*}
\bP(\hat{S}) = \sum_{S \subseteq [n]} q_S \bP(\hat{E}_S)
\overset{\eqref{a-PEJ}}{\preceq} \sum_{S \subseteq [n]} q_S | S| \Diag(\bP(\hat{E}_S)) \preceq \tau \sum_{S \subseteq [n]} q_S  \Diag(\bP(\hat{E}_S)) \overset{\eqref{eq:9hs8h9hxxx}}{=}\tau \Diag(\bP(\hat S)).
\end{equation*}

\item[(iii)] The result follows by combining the upper and lower bounds. \qedhere
\end{enumerate}
\end{proof}

In the next result we study the quantity $\lambda(\hat{S})$.

\begin{theo}  The following statements hold:
\begin{enumerate}
\item[(i)] \textbf{Lower and upper bounds.} For any sampling $\hat{S}$ we have \begin{equation}\label{eq:iuisuh77645fgh}\frac{\Exp[|\hat{S}|^2]}{n} \leq \lambda(\hat{S}) \leq \Exp[|\hat{S}|].\end{equation}
\item[(i)] \textbf{Sharper upper bound.} If $\hat{S}$ is uniform and $|\hat{S}|\leq \tau$ with probability one, then the upper bound can be improved to
\[\lambda(\hat{S}) \leq \frac{\Exp[|\hat{S}|]\tau}{n}.\]
\item[(iii)] \textbf{Identity.} If $\hat{S}$ is uniform and $|\hat{S}|=\tau$ with probability one, then 
\[\lambda(\hat{S}) = \frac{\tau^2}{n}.\]
\end{enumerate}
\end{theo}
\begin{proof}
\begin{itemize}
\item[(i)] The upper bound holds since $\lambda(\hat{S})$ is the maximal eigenvalue of $\bP(\hat{S})$ and by \eqref{eq:uniform_xx}, $\Exp[|\hat{S}|]=\Tr(\bP(\hat{S}))$. The 
lower bound follows from:
\[\lambda(\hat{S}) = \lambda(\bP(\hat{S})) \;\;\geq \;\; \frac{e^\top \bP(\hat{S}) e}{ e^\top e} \;\;\overset{\eqref{eq:is8js8sos0}}{=}\;\; \frac{\Exp[|\hat{S}|^2]}{n}.\]
\item[(ii)] By combining \eqref{eq:uniform_eigs_009} and Theorem~\ref{th-lambda} (ii) we obtain:
\[\lambda(\hat{S}) \;\;\overset{\eqref{eq:uniform_eigs_009}}{=}\;\; \frac{\Exp[|\hat{S}|]}{n}\lambda'(\hat{S})\;\;\overset{\text{Thm}~\ref{th-lambda}}{\leq}\;\; \frac{\Exp[|\hat{S}|]\tau}{n}.\]
\item[(iii)] The result follows by combining the lower bound from (i) with the upper bound in (ii).\qedhere
\end{itemize}
\end{proof}

A natural lower bound for $\lambda(\hat{S})$ (largest eigenvalue of $\bP(\hat{S})$) is  $\Exp[|\hat{S}|]/n$ (the average of the eigenvalues of $\bP(\hat{S})$). Notice that the lower bound in \eqref{eq:iuisuh77645fgh} is better than this. Moreover, observe  that both bounds in \eqref{eq:iuisuh77645fgh}  are tight. Indeed, in view of \eqref{a-PEJ}, the upper bound is achieved for any elementary sampling. The lower bound is also tight -- in view of  part (iii) of the theorem.

\subsection{Bounds for restrictions of selected samplings} \label{subsec:lambda'_restrict}

In this part we study the normalized eigenvalue associated with the   restriction of a few selected samplings (or families of samplings). In particular, we first give a (necessarily rough)  bound that holds for arbitrary  samplings,  followed by a  bound for the  $(c,\tau)$-distributed sampling  and the  $\tau$-nice sampling (both are  specific uniform samplings). Finally, we give a bound for the family of doubly uniform samplings. 

\begin{prop}\label{prop-uniform-xxjs}Let $\hat{S}$ be an arbitrary sampling and let $\tau$ be such that $|\hat{S}|\leq \tau$ with probability 1. Then for all $\emptyset \neq J\subseteq [n]$, we have
\begin{equation}\label{eq:9s898s5r768s}
\lambda'(J\cap \hat{S}) \leq \min\{|J|,\tau\}.
\end{equation}
\end{prop}
\begin{proof}\footnote{This simple result can alternatively be proved by applying \eqref{a-ABCxxx} (which we mention in a later section) together with \eqref{eq:98shsguy8658}, \eqref{a-PEJ} and the upper bound in Theorem~\ref{th-lambda}.} Note that $|J\cap \hat{S}| \leq \min\{|J|,\tau\}$ with probability 1. We only need to apply the upper bound in Theorem~\ref{th-lambda} to the restriction sampling $J\cap \hat{S}$.
\end{proof}

We now proceed to the $(c,\tau)$-distributed sampling (recall Definition~\ref{defn:distrib}).

\begin{prop}\label{prop-distrilambda}
Let $\hat S$ be the ($c,\tau$)-distributed sampling associated with a partition $\{\pP_1,\dots,\pP_c\}$ of $[n]$ such that $s=|\pP_l|$ for $l\in [c]$. Fix arbitrary $\emptyset \neq J\subseteq[n]$ and let $\omega'$ be the number of sets $\pP_l$ which have a nonempty intersection with $J$; that is, let
$\omega'\eqdef |\{l: J\cap \pP_l \neq \emptyset \}|$.
Then
\begin{equation}
\label{eq:kjhas7gbs}
\lambda'(J \cap \hat{S}) \leq  1+ \frac{(|J|-1)(\tau-1)}{s_1} + |J|
\left(\frac{\tau}{s}-\frac{\tau-1}{s_1}\right)\frac{\omega'-1}{\omega'},\end{equation}
where $s_1=\max(s-1,1)$.
\end{prop}

\begin{proof}
 By applying Lemma~\ref{lem:IIS} and Lemma~\ref{lem:su8dbd}, we get 
\begin{eqnarray}
\bP(J\cap \hat S)&\overset{\eqref{eq:98shsguy8658}}{=}&\bP(\hat E_{J})\circ \bP(\hat S) \;\;\overset{\eqref{eq:poiuyth778}}{=}\;\;\frac{\tau}{s}\left[\alpha_1 \bP(\hat E_{J})\circ \bI + \alpha_2 \bP(\hat E_{J})\circ \bE + \alpha_3 \bP(\hat E_{J})\circ(\bE-\bB)\right]\notag\\
&=&\frac{\tau}{s}\left[\alpha_1 \Diag(\bP(\hat E_{J})) + \alpha_2 \bP(\hat E_{J}) + \alpha_3 \bP(\hat E_{J})-\alpha_3\bP(\hat E_{J})\circ \bB\right].\label{a-HcaH}
\end{eqnarray}
For any $h\in \R^ n$,
\begin{equation}\label{eq:iuhdgd899}
h^\top \bP(\hat E_{J}) h= \left(\sum_{i\in J} h_i\right)^2=\left(\sum_{l=1}^c \sum_{i\in \pP_l \cap J}  h_i\right)^2\leq \omega'\sum_{l=1}^c \left( \sum_{i\in \pP_l \cap J}  h_i\right)^2
=\omega' \sum_{l=1}^c h^ \top \bP(\hat E_{J\cap \pP_l})h,
\end{equation}
where the inequality is an application of the Cauchy-Schwartz inequality. It follows that
\begin{align}\label{a-EJcricB}
\bP(\hat E_{J})\circ \bB \overset{\eqref{eq:BBB}}{=} \sum_{l=1}^ c \bP(\hat E_J)\circ \bP(\hat E_{\pP_l})=\sum_{l=1}^ c \bP(\hat E_{J\cap \pP_l}) 
\overset{\eqref{eq:iuhdgd899}}{\succeq} \frac{1}{\omega'}\bP(\hat E_J)\enspace.
\end{align}
Plugging~\eqref{a-EJcricB} into~\eqref{a-HcaH} we get:
\begin{eqnarray*}
\bP(J\cap \hat S)&\preceq &\frac{\tau}{s}\left[\alpha_1 \Diag(\bP(\hat E_{J})) + \left(\alpha_2+\alpha_3\left(1-\frac{1}{\omega'}\right)\right) \bP(\hat E_{J}) \right]\\
&\overset{\eqref{a-PEJ}}{\preceq} & \frac{\tau}{s}\left[\alpha_1+\left(\alpha_2+\alpha_3\left(1-\frac{1}{\omega'}\right)\right)|J|\right]\Diag(\bP(\hat E_J))\\
&=  & \left[ 1+ \frac{(|J|-1)(\tau-1)}{s_1} + |J|
\left(\frac{\tau}{s}-\frac{\tau-1}{s_1}\right)\frac{\omega'-1}{\omega'}\right] \Diag(\bP(\hat E_J))\circ \Diag(\bP(\hat S)).
\end{eqnarray*}
Finally, note that 
$\Diag(\bP(\hat E_J))\circ \Diag(\bP(\hat S))=\Diag(\bP(\hat E_J)\circ \bP(\hat S))\overset{\eqref{eq:98shsguy8658}}{=}\Diag(\bP(J\cap \hat S))$.
\end{proof}

We now specialize the above result to the $c=1$ case, obtaining a formula for $\lambda'(J\cap \hat{S})$ in the case when $\hat{S}$ is the $\tau$-nice sampling (recall Definition~\ref{def:tau-nice}).

\begin{prop}\label{prop-taunice}
 Let $\hat S$ be the $\tau$-nice sampling. Then for all $\emptyset \neq J\subseteq [n]$,
\begin{align}\label{a-lScapJ1}
\lambda'(J \cap \hat S )=1+ \frac{(|J|-1)(\tau-1)}{\max(n-1,1)}\enspace.
\end{align}
\end{prop}
\begin{proof}
Let $\emptyset\neq J\subseteq[n]$.
 Since $\tau$-nice sampling is the $(1,\tau)$-distributed sampling, by applying Proposition~\ref{prop-distrilambda} 
we get:
$$
\lambda'(J\cap \hat S) \leq 1+ \frac{(|J|-1)(\tau-1)}{\max(n-1,1)} \enspace.
$$
Next, by direct calculation we can verify that
\[\Exp[|J\cap \hat{S}|^2] = \frac{|J|\tau }{n}\left(1+ \frac{(|J|-1)(\tau-1)}{\max(n-1,1)}\right)\qquad \text{and} \qquad \Exp[|J\cap \hat{S}|] = \frac{|J|\tau }{n}\enspace,\]
which together with the lower bound established in Theorem~\ref{th-lambda} yields:
\[
\lambda'(J\cap \hat S)\geq \frac{\Exp[|J\cap \hat{S}|^2]}{\Exp[|J\cap \hat{S} |]}
=1+ \frac{(|J|-1)(\tau-1)}{\max(n-1,1)}\enspace.
\qedhere\]
\end{proof}

Note that \eqref{a-lScapJ1} is much better (i.e., smaller) than  the right hand side in \eqref{eq:9s898s5r768s}. This is to be expected as the bound \eqref{eq:9s898s5r768s} applies to {\em all} samplings (which have size at most $\tau$ with probability 1).

Finally, we give a bound on the normalized largest eigenvalue of the restriction of a doubly uniform sampling.

\begin{prop}\label{prop:DU-lambda'_restrict} Let $\hat{S}$ be a doubly uniform sampling which is not nil (i.e., $\Prob(\hat{S}=\emptyset)\neq 1$). Then for all $\emptyset \neq J\subseteq [n]$, 
\begin{equation}\label{eq:DU_lambda'}\lambda'(J\cap \hat{S}) \leq 1 + \frac{(|J|-1)\left(\frac{\Exp[|\hat{S}|^2]}{\Exp[|\hat{S}|]}-1\right)}{\max(n-1,1)} \enspace.\end{equation}
\end{prop}
\begin{proof}
Combining \eqref{eq:98shsguy8658} and \eqref{eq:du_prob_marix}, we get
\begin{eqnarray*}\bP(J\cap \hat{S}) &\overset{\eqref{eq:98shsguy8658}}{=} & \bP(\hat{E}_J)\circ \bP(\hat{S}) \;\; \overset{\eqref{eq:du_prob_marix}}{=}\;\;
\bE_{[J]} \circ \left(\frac{\Exp[|\hat{S}|]}{n} \left( (1 - \beta ) \bI + \beta\bE \right)\right)\\
&=& \frac{\Exp[|\hat{S}|]}{n} \left( (1 - \beta ) \bI_{[J]} + \beta\bE_{[J]} \right)\\
&\overset{\eqref{a-PEJ-equiv}}{\preceq} & \frac{\Exp[|\hat{S}|]}{n} \left( 1 - \beta   + \beta |J| \right) \bI_{[J]} \;\; = \;\; \left( 1 + (|J|-1)\beta  \right)  \Diag(\bP(J\cap \hat{S})),
\end{eqnarray*}
where $\beta$ is as in \eqref{eq:beta}.
\end{proof}

\section{Expected Separable Overapproximation}\label{sec-ESO}

In this section we develop a general technique  for computing parameters  $v=(v_1,\dots,v_n)$ for which the ESO  inequality \eqref{eq:ESOmain_def} holds.

\subsection{General technique}

We will write  $\bM_1\succeq \bM_2$ to indicate that  $\bM_1-\bM_2$ is positive semidefinite: $h^\top (\bM_1-\bM_2) h\geq 0$ for all $h\in \R^n$. It is a well known fact~\cite[Theorem 5.2.1]{HornJohnson} that the Hadamard product of two positive semidefinite matrices is positive semidefinite:
\begin{equation}\label{eq:989s8h}\bM_1 \succeq 0 \quad \& \quad \bM_2\succeq 0 \quad \Rightarrow \quad \bM_1\circ \bM_2 \succeq 0.
\end{equation}

The reason for defining and studying probability matrices $\bP(\hat S)$  is motivated by the following result, which for functions satisfying Assumption~\ref{ass:f} reduces the ESO Assumption $(f,\hat{S})\sim ESO(v)$ to the problem of bounding  the Hadamard product of the probability matrix $\bP(\hat{S})$ and the data matrix $\bA^\top \bA$ from above by a diagonal matrix. Note that because $\bP(\hat{S})\succeq 0$,  in view of \eqref{eq:989s8h},the Hadamard product $\bP(\hat{S})\circ \bA^\top \bA$ is positive semidefinite.

\begin{lemma}\label{lem-Pa}
If $f$ satisfies Assumption~\ref{ass:f} and 
\begin{align}\label{a-PAAwp}
\bP(\hat S)\circ (\bA^\top \bA)
\preceq \Diag(v\circ p), 
\end{align}
for some vector $v\in \R^{n}_{++}$,
where $p$ is the vector of probabilities defined in~\eqref{eq:p_i},
then \[(f,\hat{S})\sim ESO(v).\] 
\end{lemma}

\begin{proof} Let us substitute $h\leftarrow h_{[\hat{S}]}$ into \eqref{eq:shs7hs8} and take expectation in $\hat{S}$ of both sides. Applying \eqref{eq:sjdud7xxx}, we obtain:
\begin{align}\label{a-preEso}
 \E[f(x+h_{[\hat S]})]\leq f(x)+ \<\Diag(\bP(\hat{S}))\nabla f(x),h > +\frac{1}{2} h^ \top (\bP(\hat S) \circ (\bA^\top \bA)) h, \qquad \forall x,h\in \R^ n.
\end{align}
It remains to apply  assumption \eqref{a-PAAwp}.
\end{proof}

We next focus on the problem of finding vector $v$ for which \eqref{a-PAAwp} holds. The following direct consequence of \eqref{eq:989s8h} will be helpful in this regard:
 \begin{align}\label{a-ABC}
( 0\preceq \bM_1 \quad  \&  \quad  \bM_2 \preceq \bM_3 ) \quad \Rightarrow \quad \bM_1\circ \bM_2 \preceq \bM_1 \circ \bM_3.
 \end{align}
In particular, \eqref{a-ABC} can be used to establish the first part of the following useful lemma.

\begin{lemma}\label{l-PaTA} If  $\bM_1 \succeq 0$ and $\bM_2\succeq 0$, then
 \begin{eqnarray}
 \lambda'(\bM_1 \circ \bM_2) &\leq & \min\{ \lambda'(\bM_1),\lambda'(\bM_2)\}, \label{a-ABCxxx}\\
\lambda'(\bM_1 + \bM_2) &\leq & \max\{ \lambda'(\bM_1),\lambda'(\bM_2) \}. \label{a-ABCxxxyyy}
\end{eqnarray}
\end{lemma}
\begin{proof}
By definition, $ \bM_2 \preceq \lambda'(\bM_2) \Diag(\bM_2)$,
which together with \eqref{a-ABC} implies:
$$
\bM_1 \circ \bM_2 \preceq  \lambda'(\bM_2) \left(\bM_1 \circ \Diag(\bM_2)\right)=\lambda'(\bM_2) \Diag(\bM_1 \circ \bM_2).
$$
Applying the same reasoning to the matrix $\bM_1$ we obtain:
$\bM_1\circ \bM_2 \preceq \lambda'(\bM_1)\Diag(\bM_1 \circ \bM_2).$ Combining the two results, we obtain \eqref{a-ABCxxx}. Inequality \eqref{a-ABCxxxyyy} follows from: \begin{eqnarray*}
\bM_1 + \bM_2 &\preceq& \lambda'(\bM_1)\Diag(\bM_1) + \lambda'(\bM_2)\Diag(\bM_2) \\
&\leq & \max\{\lambda'(\bM_1),\lambda'(\bM_2)\}(\Diag(\bM_1)+\Diag(\bM_2)) \\
&=& \max\{\lambda'(\bM_1),\lambda'(\bM_2)\} \Diag(\bM_1+\bM_2).  \end{eqnarray*}
\end{proof}

\subsection{ESO I: no coupling between the sampling and data}

By applying Lemma~\ref{l-PaTA}, Eq \eqref{a-ABCxxx}, to  $\bM_1=\bP(\hat{S})$ and $\bM_2=\bA^\top \bA$,  we obtain a formula for $v$ satisfying \eqref{a-PAAwp}.

\begin{theo}[ESO without coupling between sampling and data]\label{prop-w1} Let $f$ satisfy Assumption~\ref{ass:f} and 
let  $\hat S$ be an arbitrary sampling. Then $(f,\hat{S})\sim ESO(v)$
for $v=(v_1,\dots,v_n)$ defined by
\begin{align}\label{a-w1}
v_i=\min\{\lambda'(\bP(\hat{S})),\lambda'(\bA^\top \bA)\}\sum_{j=1}^ m \bA_{ji}^2,\qquad  i\in [n].
\end{align}
\end{theo}
\begin{proof}
Let $\bP = \bP(\hat{S})$. To establish the main statement, it is sufficient to apply Lemma~\ref{lem-Pa} and Lemma~\ref{l-PaTA} and note that for $v$ defined by~\eqref{a-w1},
$\Diag(v\circ p)=\min(\lambda'(\bP),\lambda'( \bA^\top \bA)) \Diag(\bP\circ \bA^\top \bA)$. 
\qedhere
\end{proof}

If for some $\tau$, $|\hat{S}|\leq \tau$ with probability 1, then  in view of Theorem~\ref{th-lambda}, we have $\lambda'(\bP(\hat{S}))\leq \tau$. Furthermore, 
\[\lambda'(\bA^\top \bA) = \lambda' \left( \sum_{j=1}^m \bA_{j:}^\top \bA_{j:} \right) \overset{\eqref{a-ABCxxxyyy}}{\leq} \max_{j}\lambda'(\bA_{j:}^\top \bA_{j:}) \overset{\eqref{eq:09s09shjyu}}{=} \max_j \|\bA_{j:}\|_0 = \max_j |J_j|.\]

Hence, in view of Lemma~\ref{prop-w1}, we can pick the ESO parameter conservatively as follows:
\begin{equation}\label{eq:si98gsssf}v_i \leq \min\{\tau,\max_j |J_j|\} \sum_{j=1}^m \bA_{ji}^2, \quad i\in [n].\end{equation}
An ESO inequality with $v_i$ similar to \eqref{eq:si98gsssf}  was established in \cite{PCDM}, but for a different class of functions ($\omega$-partially separable functions: functions expressed as a sum of functions each of which depends on at most $\omega$ coordinates) and uniform samplings only. Indeed, the bound established therein for arbitrary uniform samplings uses $v_i = \min\{\tau,\omega\}L_i$, where $\omega$ is the degree of separability of $f$ and $L_i$ is the Lipschitz constant of $\nabla f$ associated with coordinate $i$. In our setting, $\omega = \max_j |J_j|$ and $L_i$ corresponds to $\sum_j \bA_{ji}^2$. Hence, \eqref{eq:si98gsssf} could be seen as a generalization of the ESO bound in \cite{PCDM} to arbitrary samplings.

Note that  computation of
 the normalized eigenvalue $\lambda'(\bA^\top \bA)$ could be time-consuming, and would require a number of passes through the data prior to running a coordinate descent method, which may be prohibitive. In the next section we follow a different approach, one in which this issue is avoided. The main idea is to decompose $\bA^\top \bA$ as a sum of the rank one matrices $\bA_{j:}^\top \bA_{j:}$ and then bound each term $\bP(\hat{S})\circ \bA_{j:}^\top \bA_{j:}$ separately.


\subsection{ESO II: coupling the sampling with data} \label{sec:ESO2}

In this section we use a different strategy for satisfying \eqref{a-PAAwp}. We first write 
\[\bP(\hat{S})\circ \bA^\top \bA =\bP(\hat{S})\circ \sum_{j=1}^m \bA_{j:}^\top \bA_{j:} = \sum_{j=1}^m \bP(\hat{S})\circ \bA_{j:}^\top \bA_{j:},\] where $\bA_{j:}$ denote the $j$th row vector of matrix $\bA$
and then bound each term in the last sum individually.   Recall the definition of set $J_j$ from \eqref{a-Jj}: $J_j =\{i\in [n] \;:\; \bA_{ji}\neq 0\}$.

\begin{theo}[ESO with coupling between sampling and data]
\label{thm:ESO}
Let  $\hat{S}$ be an arbitrary  sampling   and $v=(v_1,\dots,v_n)$ be defined by: 
\begin{equation}\label{eq:w_i}v_ i  = \sum_{j=1}^m \lambda'(J_j\cap \hat{S}) \bA_{ji}^2, \quad i=1,2,\dots,n.\end{equation}
Then $(f,\hat{S})\sim ESO(v)$.
\end{theo}
\begin{proof} 
Let $j\in[m] $ and $\bA_{j:}$ denote the $j$th row vector of matrix $\bA$. By the definition of $J_j$, 
\[\bA_{j:}^\top \bA_{j:}=(e_{[J_j]} e_{[J_j]}^\top)\circ (\bA_{j:}^\top \bA_{j:})=\bP(\hat E_{J_j})\circ(\bA_{j:}^\top \bA_{j:}).\]
Thus, $
\bP(\hat S)\circ (\bA_{j:}^\top \bA_{j:})=\bP(\hat S)\circ\bP(\hat E_{J_j})\circ(\bA_{j:}^\top \bA_{j:})=\bP(J_j\cap \hat S)\circ (\bA_{j:}^\top \bA_{j:})$. We now apply Lemma~\ref{l-PaTA} to the sampling $J_j\cap \hat S$ and the matrix $\bA_{j:}$ and obtain:
\begin{equation}
\label{a-bPScnb}
\begin{split}
\bP(\hat S)\circ (\bA_{j:}^\top \bA_{j:}) &\preceq \min\{\lambda'(J_j\cap \hat S), \lambda'(\bA_{j:}^\top \bA_{j:})\} \Diag(\bP(\hat S)\circ (\bA_{j:}^\top \bA_{j:}))\\
&\preceq \lambda'(J_j\cap \hat S) \Diag(\bP(\hat S)\circ (\bA_{j:}^\top \bA_{j:})).
\end{split}
\end{equation}
Therefore,
\begin{eqnarray*}
\bP(\hat S )\circ \bA^\top \bA & = &
\sum_{j=1}^m \bP(\hat S)\circ (\bA_{j:}^\top \bA_{j:}) \;\; \preceq \;\; \sum_{j=1}^m \lambda'(J_j\cap \hat S) \Diag(\bP(\hat S)\circ (\bA_{j:}^\top \bA_{j:}))\\&=&
\Diag(\bP(\hat S))\circ \sum_{j=1}^m \lambda'(J_j\cap \hat S)\Diag(\bA_{j:}^\top \bA_{j:}) \;\; =\;\; \Diag(p)\circ \Diag(v),
\end{eqnarray*}
where $p=(p_1,\dots,p_n)$ is the  vector of probability defined in~\eqref{eq:p_i} and $v=(v_1,\dots,v_n)$ is defined in~\eqref{eq:w_i}.  For completeness, let us show that the second inequality in~\eqref{a-bPScnb} can be replaced by equality. Indeed, from \eqref{eq:09s09shjyu} and the fact that $|J_j| = \|A_{j:}\|_0$, we obtain $\lambda'(\bA_{j:}^\top \bA_{j:})=|J_j|$. Finally, using the upper bound in Theorem~\ref{th-lambda}, we know that $\lambda'(J_j \cap \hat S)\leq |J_j|$. Hence,
$\min\{ \lambda'(J_j\cap \hat S), \lambda'(\bA_{j:}^\top \bA_{j:} )\}= \lambda'(J_j\cap \hat S)$.
\end{proof}

The benefit of this approach is twofold: First, if the data matrix $\bA$ is sparse, the sets $J_j$ have small cardinality, and  from Proposition~\ref{prop-uniform-xxjs} (or other results in Section~\ref{subsec:lambda'_restrict}, depending on the sampling $\hat{S}$ used) we conclude that $\lambda'(J_j\cap \hat{S})$ is small. Hence,  the parameters $v_i$ obtained through \eqref{eq:w_i} get better (i.e., smaller) with  sparser data. Second, the formula for $v_i$ does not involve the need to compute an eigenvalue associated with the data matrix. On the other hand, instead of having to compute $\lambda'(\hat{S})$ (which, as we have seen, is equal to $\tau$ if $|\hat{S}|=\tau$ with probability 1), we now need to compute the normalized largest eigenvalue of $m$ restrictions of $\hat{S}$, $\lambda'(J_j\cap \hat{S})$ for all $j=1,2,\dots,m$. However,  for this there is a good upper bound available through Proposition~\ref{prop-uniform-xxjs} for an arbitrary sampling, and refined bounds can be derived for specific samplings (for examples, see Section~\ref{subsec:lambda'_restrict}).

\subsection{ESO without eigenvalues} \label{sec:applications}

In this section we illustrate the use of the techniques developed in the preceding sections to derive  ESO inequalities, for selected samplings, which do not depend on any eigenvalues, and lead to easily computable ESO parameters $v=(v_1,\dots,v_n)$. The techniques can be used to derive similar ESO inequalities for other samplings as well. 

\begin{prop}\label{prop:98ys98sh8s}
Let $f$ satisfy Assumption~\ref{ass:f} and let sets $J_1,\dots,J_m$ be defined as in \eqref{a-Jj}.  Then $(f,\hat{S})\sim ESO(v)$ provided that the sampling $\hat{S}$ and vector $v$ are chosen in any of the  following ways:
\begin{enumerate}
\item[(i)]  $\hat{S}$ is an arbitrary sampling such that $|\hat{S}|\leq \tau$ with probability 1, and
\begin{equation}\label{eq:sig798sg9ss}v_i = \sum_{j=1}^m \min\{|J_j|,\tau\}\bA_{ji}^2, \quad i=1,2,\dots,n.\end{equation}
\item[(ii)]$\hat S$ is the $(c,\tau)$-distributed sampling and
\begin{equation}\label{eq:w_i_08yb89v8ujk}
v_ i  = \sum_{j=1}^m \left[ 1+ \frac{(|J_j|-1) (\tau-1)}{s_1} + |J_j|
\left(\frac{\tau}{s}-\frac{\tau-1}{s_1}\right)\frac{\omega_j'-1}{\omega_j'}\right] \bA_{ji}^2, \quad i=1,2,\dots,n,
\end{equation}
where  
$ \omega_j'\eqdef |\{l: \pP_l\cap J_j \neq 0\}|$ for $j\in [m]$.
\item[(iii)] $\hat S$ is the $\tau$-nice sampling (for $\tau\geq 1$) and
\begin{equation}\label{eq:w_i_xxsdda}
v_ i  = \sum_{j=1}^m \left[1+ \frac{(|J_j|-1)(\tau-1)}{\max(n-1,1)}\right] \bA_{ji}^2, \quad i=1,2,\dots,n,
\end{equation}
\item[(iv)] $\hat S$ is a doubly uniform sampling (which is not nil) and
\begin{equation}\label{eq:w_i_xxsddagggggg}
v_ i  = \sum_{j=1}^m \left[1+ \frac{(|J_j|-1)\left(\frac{\Exp[|\hat{S}|^2]}{\Exp [|\hat{S}|]}-1\right)}{\max(n-1,1)}\right] \bA_{ji}^2, \quad i=1,2,\dots,n,
\end{equation}
\item[(v)]  $\hat S$ is a graph sampling and
\begin{equation}\label{eq:coro-sep-w_i}
v_ i  =  \sum_{j=1}^m \bA_{ji}^2, \quad i=1,2,\dots,n.
\end{equation}
\item[(vi)] $\hat S$ is a serial sampling (i.e., a sampling for which $|\hat{S}|=1$ with probability 1) and
$v=(v_1,\dots,v_n)$ is defined as in~\eqref{eq:coro-sep-w_i}.
\end{enumerate}
\end{prop}

\begin{proof}
\begin{enumerate}
\item[(i)]
A direct consequence of Theorem~\ref{thm:ESO} and  Proposition~\ref{prop-uniform-xxjs}.
\item[(ii)] A direct consequence of Theorem~\ref{thm:ESO} and  Proposition~\ref{prop-distrilambda}.
\item[(iii)] This is a special case of part (ii) for $c=1$. 
\item[(iv)] A direct consequence of Theorem~\ref{thm:ESO} and  Proposition~\ref{prop:DU-lambda'_restrict}.
\item[(v)] For a graph sampling it is clear that $|J_j \cap \hat S|\leq 1$ 
with probability 1 for all $j\in[m]$. The result then follows  from Theorem~\ref{thm:ESO}.
\item[(vi)]  A special case of (v). Indeed, a single vertex is  an independent set of a graph.
\end{enumerate}
\end{proof}

{\em Remarks:} Note that part (i) of Proposition~\ref{prop:98ys98sh8s}  is a strict improvement on  \eqref{eq:si98gsssf}. Also, this is strict improvement, both in the quality of the bound and in generality of the sampling, on the result in \cite{PCDM}, which was proved for uniform samplings only and where the bound involved $\max_j |J_j|$ instead of $|J_j|$.   Part (ii) should be compared with the results  obtained in \cite{Hydra2} and part (iii)  with those in \cite{APPROX, PCDM}.

\section{Discussion}

\subsection{Trade-off between preprocessing time and iteration complexity}

As stressed before, smaller parameter $v=(v_1,\dots,v_n)$ leads to better convergence result (see Table~\ref{tbl:complexity}) but computing the smallest admissible $v$ would require too large computational effort.
Nevertheless, using a cheaply computed parameter $v=(v_1,\dots,v_n)$ would lead to  large iteration complexity and 
 slow convergence.
The trade-off between the preprocessing time for computing 
the parameter $v=(v_1,\dots,v_n)$ and the iteration complexity of the algorithm shall be discussed next.

For specific samplings such as $\tau$-nice sampling and $(c,\tau)$-distributed sampling, admissible $v$ can be computed using dedicated formulae~\ref{eq:w_i_08yb89v8ujk} and~\ref{eq:w_i_xxsdda}, which appeared respectively in~\cite{Hydra2} and~\cite{APPROX}. 
For arbitrary sampling $\hat S$, admissible parameter $v$ can be computed  according to~\ref{eq:si98gsssf},~\ref{eq:w_i} or~\ref{eq:sig798sg9ss}, which are given for the first time. While~\ref{eq:w_i} requires computing the largest 
eigenvalue for 
$m$ matrices of sizes $\{J_1,\dots,J_m\}$, both~\ref{eq:si98gsssf} and~\ref{eq:sig798sg9ss}
can be computed in at most two passes over the data. In return, ~\ref{eq:w_i} provides a smaller parameter 
$v$ which improves the iteration complexity. 

For approximating $\lambda'(J_j\cap \hat S)$, one can apply power method on the positive semidefinite matrix $\bP(J_j  \cap \hat S)$.
The number of operations needed in one iteration of the power method is $|J_j|^2$ and if we apply $T$ iterations of power 
method\footnote{ Note that as 
the matrix $\bP(J_j\cap \hat S)$ is positive semidefinite, the power method always converges to the largest eigenvalue even if it is 
not a dominant eigenvalue. We defer the study on the convergence rate of power method for different matrices $\bP(J_j\cap \hat S)$ to 
a future work. }, then the total number of operations needed for computing $v$ using~\ref{eq:w_i} is
$$
O(T\sum_{j=1}^ m |J_j|^2) \leq O( T\max_{j}|J_j| \nnz(\bA)),
$$ 
where the big $O$ notation hides constants independent of the data matrix $\bA$.

\begin{table}[!t]
\centering
\begin{tabular}{|c|c|}
\hline
$v=(v_1,\dots,v_n)$ & Number of passes over the data  \\
\hline
&\\
~\ref{eq:si98gsssf}
 & $\displaystyle O(1+\frac{1}{\lambda}\max_{i}\frac{v_i\tau}{p_i n}\log(\frac{1}{\epsilon}))$\\
&\\ \hline & \\
~\ref{eq:w_i} & $\displaystyle O(\frac{T\sum_{j=1}^m |J_j|^2}{\nnz(\bA)} +\frac{1}{\lambda}\max_{i}\frac{v_i\tau}{p_i n}\log(\frac{1}{\epsilon}))$ \\
&\\  \hline & \\
~\ref{eq:sig798sg9ss} & $\displaystyle O(1+\frac{1}{\lambda}\max_{i}\frac{v_i\tau}{p_i n}\log(\frac{1}{\epsilon}))$ \\
&\\ \hline
\end{tabular}
\caption{Total number of passes over data for three different admissible parameters $v$. }
\label{tbl:scom}
\end{table}

Recall from Table~\ref{tbl:complexity} how the iteration complexity of different methods depends on the parameter $v=(v_1,\dots,v_n)$. 
Let us consider the strongly convex smooth objective function setup and assume that the random sampling
$\hat S$ has cardinality $\tau$ with probability 1.  Then the computational time of one epoch ($n$ iterations)
is of the same order as $\tau$ passes over the data. Therefore, given a parameter $v=(v_1,\dots,v_n)$, the number of passes over the data is bounded by:
$$
O(\frac{1}{\lambda}\max_{i}\frac{v_i\tau}{p_i n}\log(\frac{1}{\epsilon})),
$$
where $\epsilon$ is the target accuracy and $\lambda$ is the strong convexity parameter of the problem.

The comparison of the three formulae in terms of overall complexity is reported in Table~\ref{tbl:scom}, where the big $O$
notation hides constants independent of the data matrix $\bA$. 
It is clear from the table that the trade-off between the preprocessing and the iteration complexity mainly
depends on the proportion between $\tfrac{T\sum_{j=1}^m |J_j|^2}{\nnz(\bA)}$ and $\frac{1}{\lambda} \max_{i} \frac{v_i\tau}{p_i n}$.
In Table~\ref{tbl:experiments} we report the actual computing time of $v$ using different formulae and the corresponding value of $\max_{i}\frac{v_i\tau}{p_i n}$,
for two real data matrices w8a and dorothea. To facilitate the comparison we normalized the two data sets so that
the diagonal elements of $\bA^\top \bA$ are all one. The samplings $\hat S$ that we used in the experiments are all 
 product sampling
(Definition~\ref{def:prosam})
with respect to some random partition of the set $[n]$.  The number of iterations $T$ for the power method
is fixed to 10 and we multiply the obtained value by 1.01. 
Because of the comparable processing time, Formula~\ref{eq:sig798sg9ss} is clearly better than Formula~\ref{eq:si98gsssf}.
From Table~\ref{tbl:experiments} 
we also see that Formula~\ref{eq:w_i} requires significant computational effort for computing $v$ comparing to the other two but also
 reduces 
the value of $\max_{i} \frac{v_i\tau}{p_i n}$ by order of magnitude in most of the regimes. Let us take the example of
dorothea with $\tau=256$, then the overall number of passes over data is 
$O(1.52+\frac{256.91}{\lambda}\log(\frac{1}{\epsilon}))$ if $v$ is computed 
using Formula~\ref{eq:sig798sg9ss}  and $O(8715.8+\frac{16.68}{\lambda}\log(\frac{1}{\epsilon}))$ if $v$
is computed 
using Formula~\ref{eq:w_i}. Hence for small enough strong convexity parameter $\lambda$, it is worth to spend more time in computing 
a good parameter $v$ using Formula~\ref{eq:w_i}, which will then be compensated by a smaller iteration complexity.


\begin{table}[!t]
\centering
        \begin{tabular}{|c|c|c|c|c|c|c|c|}
            \hline
             \multirow{2}{*}{ Data} & \multirow{2}{*}{ $\tau$}& \multicolumn{3}{ c|}{$\frac{\mathrm{time~of~computing~} v}{\mathrm{time~of~one~pass~over~data}}$} & \multicolumn{3}{ c|}{$\max_{i}\frac{v_i\tau}{p_in}$} \\
            \cline{3-8}
         &  &  Form.~\ref{eq:si98gsssf} & Form.~\ref{eq:sig798sg9ss}&    Form.~\ref{eq:w_i}  & Form.~\ref{eq:si98gsssf}  & Form.~\ref{eq:sig798sg9ss} & Form.~\ref{eq:w_i} \\
            \hline
         &   1& 1 & 1 & 1 &1 & 1 & 1 \\
            \cline{2-8}
  w8a & 8 & 1 &  1.9 &382.5 & 8.11  & 8.11& 1.92  \\
           \cline{2-8}
  {$n=300$}  &16 &1  & 1.8& 380.3 & 17.07  & 17.10 & 3.03 \\
  \cline{2-8}
 {$m=49749$}&24 & 1  & 2.0&377.2 & 24.96  & 24.97 &3.98 \\
  \cline{2-8}
 {$\frac{\nnz}{nm}\simeq 3.8\%$} &128 & 1  & 1.8 & 331.4& 145.92 & 50.45  & 20.80 \\
  \cline{2-8}
  &256 & 1  & 1.8& 313.2 & 194.56  & 67.13& 50.54 \\
           \hline
            &   1& 1 & 1 & 1 &1 & 1 & 1 \\
            \cline{2-8}
  dorothea & 8 & 1  & 3.2& 8057.1 & 8.01 & 8.01  & 1.44 \\
           \cline{2-8}
  {$n=100000$}  &16 &1  & 1.52& 8442.4 & 16.01 & 16.01 & 1.93  \\
  \cline{2-8}
 {$m=800$}&24 & 1 & 1.52 &8546.5 & 24.01  & 24.01 &2.42 \\
  \cline{2-8}
 {$\frac{\nnz}{nm}\simeq 0.91\%$} &128 & 1  & 1.52 & 8686.6& 128.13  & 128.13& 8.79  \\
  \cline{2-8}
  &256 & 1  & 1.52 & 8715.8& 256.91 & 256.91& 16.68  \\
  \cline{2-8}
  &1024 & 1  & 1.53 & 8724.1& 1038.1& 1038.1 & 64.5  \\
           \hline
   \end{tabular}
\caption{Comparison of  Formula~\ref{eq:si98gsssf}, Formula~\ref{eq:w_i}  and Formula~\ref{eq:sig798sg9ss}. }
\label{tbl:experiments}
\end{table}

\subsection{Optimal sampling}

Proposition~\ref{prop:98ys98sh8s} should be understood in the context of complexity results for randomized coordinate descent, such as those in Table~\ref{tbl:complexity}. For instance, in view of \eqref{eq:sig798sg9ss} for an arbitrary sampling $\hat{S}$ such that $|\hat{S}|\leq \tau$ with probability 1, the accelerated coordinate descent method developed in \cite{Paper1} has complexity
\begin{equation}\label{eq:98gs98s}\sqrt{2\sum_{i=1}^n \frac{v_i (x^0_i-x^*_i)^2}{p_i^2}} \times \frac{1}{\sqrt{\epsilon}} = \sqrt{2\sum_{i=1}^n \frac{\sum_{j=1}^m \min\{|J_j|,\tau\}\bA_{ji}^2 (x^0_i-x^*_i)^2}{p_i^2}} \times \frac{1}{\sqrt{\epsilon}} .\end{equation}
Naturally, the bound improves if we use a specialized sampling, such as the $\tau$-nice sampling (since the constants $v_i$ become smaller). 

Sometimes, one can find a sampling which minimizes the complexity bound. For instance, if we restrict our attention to serial samplings only (samplings picking  a single coordinate at a time), then one can find probabilities $p_1,\dots,p_n$, which uniquely define a sampling, minimizing the complexity bound:
\begin{equation}\label{eq_optimal_sampling} p_i = \frac{\left(w_i (x_i^0-x_i^*)^2\right)^{1/3}}{\sum_{i=1}^n \left(w_i (x_i^0-x_i^*)^2\right)^{1/3}}, \qquad i\in [n],\end{equation}
where $w_i = \sum_{j} \bA_{ji}^2$. Note that if the $i$th coordinate is optimal at the starting point (i.e., if $x_i^0=x_i^*$), then the prediction is to choose $p_i=0$ (i.e., to never update coordinate $i$) -- this is what one would expect. Using the serial sampling defined by \eqref{eq_optimal_sampling}, the complexity \eqref{eq:98gs98s}  takes the form
\[C_{opt} = \sqrt{2}\left(\sum_{i=1}^n w_i^{1/3} (x_i^0 - x_i^*)^{2/3}\right)^{3/2}\times \frac{1}{\sqrt{\epsilon}} = \frac{\sqrt{2}\|d\|_{2}^3}{\sqrt{\epsilon}},\]
where $d\in \R^n$ with $d_i = w_i^{1/6}(x_i^0-x_i^*)^{1/3}$ and $\|d\|_{q} = (\sum_{i=1}^n d_i^{q})^{1/q}$. However, if the uniform serial sampling is used instead (each coordinate is chosen with probability $p_i=1/n$), then the complexity \eqref{eq:98gs98s} has the form
\[C_{unif} = \sqrt{2} n \left(\sum_{i=1}^n w_i (x_i^0 - x_i^*)^2 \right)^{1/2} \times \frac{1}{\sqrt{\epsilon}} = \frac{\sqrt{2}n \|d\|_6^3}{\sqrt{\epsilon}}.\]

While $\|d\|_6\leq \|d\|_2$ for all $d$, these quantities can be equal, in which case $C_{opt}$ is $n$ times better than $C_{unif}$.

%

\section{Conclusion}\label{sec:conclusion}

We have conducted a systematic study of ESO inequalities for a large class of functions (those satisfying Assumption~\ref{ass:f}) and {\em arbitrary} samplings. These inequalities are crucial in the design and complexity analysis of randomized coordinate descent methods. This led us to study  standard and normalized largest eigenvalue of the Hadamard product of the probability matrix associated with a sampling and a certain positive semidefinite matrix containing the data defining the function. Using our approach we have established new ESO results and also re-derived  ESO results
already established in  the literature (in the case of uniform samplings) via different techniques. Our approach can be used to derive further bounds for specific samplings and can potentially be of interest outside the domain of randomized coordinate descent.

\bibliographystyle{plain}
\bibliography{biblio}

\clearpage

\appendix

\section{Frequently used notation}

\begin{table}[!h]
\begin{center}
\begin{tabular}{|c|l|c|}
 \hline
\multicolumn{3}{|c|}{{\bf Samplings}}\\
\hline
$\Prob$ & Probability  & \\
$\Exp$ & Expectation & \\
$S, J$ & subsets of $[n]\eqdef \{1,2,\dots,n\}$ & \\
$\hat{S}$ & sampling, i.e., a random subset of $[n]$ &  Sec~\ref{sec:introESO},\ref{subsec:samplings}\\
$\hat{E}_S$ & elementary sampling associated with set $S\subseteq [n]$ & Def~\ref{df:elementary}\\
$\hat{S}_1\cap \hat{S}_2$ & intersection of samplings $\hat{S}_1$ and $\hat{S}_2$  & Def~\ref{def:intersection}\\
$J\cap \hat{S}$ & restriction of sampling $\hat{S}$ to set $J$ (=$\hat{E}_S\cap \hat{S}$) & Def~\ref{def:restriction} \\
$\bP = \bP(\hat{S})$ & $n$-by-$n$ probability  matrix:  $\bP_{ij} = \Prob(\{i,j\}\subseteq \hat{S})$ & Sec~\ref{sec:prob}  \\ 
$p_i$ & $p_i= \bP_{ii} = \Prob(i \in \hat{S}) $ & \eqref{eq:p_i}\\
$p$ & $p =  (p_1,\dots,p_n)^\top \in \R^n$ & \eqref{eq:p_i}\\
\hline
\multicolumn{3}{|c|}{{\bf Matrices and vectors}} \\
\hline
$e$ & the $n$-by-$1$ vector of all ones & \\
$e_i$ & the $i$-th unit coordinate vector in $\R^n$ & \\
$h_{[S]}$ & for $h\in \R^n$ and $S\subseteq [n]$, this is defined by $h_{[S]}=\sum_{i\in S} h_i e_i$ & \\
$\bA$ & $m$-by-$n$ data matrix defining $f$ & \eqref{eq:shs7hs8}\\
$J_j$ & the set of $i\in [n]$ for which $\bA_{ij}\neq 0$ & \eqref{a-Jj} \\
$\bI$  & $n$-by-$n$ identity matrix & \\
$\bE$ & $n$-by-$n$ matrix of all ones & \\
$\Diag$ & outputs a diagonal matrix based on its argument (matrix or vector) & \\
$ \circ $ & Hadamard (elementwise) product of two  matrices or vectors &  \\
$\bM_{[S]}$ & restriction of matrix $\bM\in \R^{n\times n}$ to rows and columns indexed by $S$  & \eqref{eq:sig98s445566} \\
$\lambda(\bM)$ & maximal eigenvalue of $n$-by-$n$ matrix $\bM$ & Sec~\ref{sec:eig}; \eqref{eq:lambdaM} \\
$\lambda'(\bM)$ & normalized maximal eigenvalue of $n$-by-$n$ matrix $\bM$ & Sec~\ref{sec:eig}; \eqref{eq:lambda'M} \\
$\lambda'(\hat{S})$ & shorthand notation for $\lambda'(\bP(\hat{S}))$ & \\
\hline
\end{tabular}
\end{center}
\caption{Notation appearing frequently in the paper.}
\label{tbl:notation}
\end{table}

\end{document}